\documentclass[review,11pt,authoryear]{elsarticle}
\usepackage{amsfonts}
\usepackage{amssymb}
\usepackage{amsthm}
\usepackage{subcaption}
\usepackage{graphicx}
\usepackage{color}
\usepackage{amsmath}
\usepackage[english]{babel}
\usepackage{lineno}
\usepackage{xcolor}
\usepackage{ulem}

\journal{European Journal of Operational Research}
\newtheorem{proposition}{Proposition}
\newtheorem{remark}{Remark}
\newtheorem{definition}{Definition}
\newtheorem{lemma}{Lemma}
\newtheorem{theorem}{Theorem}
\newtheorem{corollary}{Corollary}
\newtheorem{example}{Example}
\newcommand{\black}{\color{black}}

\definecolor{OliveGreen}{rgb}{0,0.6,0}

\begin{document}

%Para las gráficas, los puntos van a grosor 2, el eje x a grosor 14 y los ejes a grosor 12.

\begin{frontmatter}

%% Title, authors and addresses

%% use the tnoteref command within \title for footnotes;
%% use the tnotetext command for theassociated footnote;
%% use the fnref command within \author or \address for footnotes;
%% use the fntext command for theassociated footnote;
%% use the corref command within \author for corresponding author footnotes;
%% use the cortext command for theassociated footnote;
%% use the ead command for the email address,
%% and the form \ead[url] for the home page:
%% \title{Title\tnoteref{label1}}
%% \tnotetext[label1]{}
%% \author{Name\corref{cor1}\fnref{label2}}
%% \ead{email address}
%% \ead[url]{home page}
%% \fntext[label2]{}
%% \cortext[cor1]{}
%% \address{Address\fnref{label3}}
%% \fntext[label3]{}

\title{Stochastic comparisons of imperfect maintenance models for a gamma deteriorating system}

%% use optional labels to link authors explicitly to addresses:
%% \author[label1,label2]{}
%% \address[label1]{}
%% \address[label2]{}

\author[Mercier]{Sophie Mercier}
\address[Mercier]{CNRS/Univ Pau \& Pays Adour, IPRA-LMAP, 64000 PAU, France\\ Corresponding author, sophie.mercier@univ-pau.fr}

\author[Castro]{I.T. Castro}
\address[Castro]{Department of Mathematics, University of Extremadura, Spain -  inmatorres@unex.es}

\begin{abstract}
This paper compares two imperfect repair models for a degrading system, with deterioration level modeled by a non homogeneous gamma process. Both models consider instantaneous and periodic repairs. The first model assumes that a  repair reduces the degradation of the system accumulated from the last maintenance action. The second model considers a virtual age model and assumes that a repair reduces the age accumulated by the system since the last maintenance action. Stochastic comparison results between the two resulting processes are obtained. Furthermore, a specific case is analyzed, where the two repair models provide identical expected deterioration levels at maintenance times. Finally, two optimal maintenance strategies are explored, considering the two models of repair. 
\end{abstract}


\end{frontmatter}

\section{Introduction}

{\black Safety and dependability are crucial issues in many industries, which have  lead to the development of a huge literature devoted to the so-called reliability theory. In the oldest literature, the lifetimes of industrial systems or components were usually directly modeled through random variables, see, e.g., \cite{BarlowProschan1965} for a pioneer work on the subject. In case of repairable systems, successive lifetimes of a system then appear as the points of a counting process leading to so-called recurrent events. Based on the development of on-line monitoring which allows the effective measurement of a system deterioration (length of a crack, thickness of a cable, intensity of vibrations, temperature, ...), numerous papers nowadays model the degradation in itself, which is often considered to be (mostly) monotonous with respect to the time. This is done through the use of stochastic processes such as Wiener processes with trend (\cite{LiuWuXieKuo2017,ZhangGaudoinXie2015,HuLeeTang2015}), inverse gaussian models (\cite{ChenYeXiangZhang2015}), transformed Beta degradation processes (\cite{GiorgioPulcini2018}) or gamma processes (\cite{HuynhCastroBarrosBerenguer2014}), among others (see also \cite{MercierPham2012} in case of a bivariate deterioration indicator). This paper focuses on gamma processes, which seem the most popular, see \cite{Noortwijk2009} with a large number of references therein}.

To mitigate the effect of the system degradation and to extend the system lifetime, a large volume of maintenance models have been proposed in the literature. Most of these models are limited to perfect repairs (\cite{Caballe2015,HongZhouZhangYe2014}). However, imperfect maintenance actions describe more realistic situations than perfect repairs. Some advances have been made to include imperfect
repairs in a degrading system (\cite{Alaswad2017,GiorgioPulcini2018}). However, as \cite{ZhangGaudoinXie2015} claimed, the issue of treating imperfect maintenance in the context of degrading systems remains widely open nowadays.
%The
%first type assumes that the maintenance actions return the system to a
%previous stage of deterioration. In the second type, the imperfect
%maintenance reduces the degradation level of the maintained system (\cite{Castro2016}). The third approach is based on the idea
%that the maintenance action changes the rate of degradation of the system (%
%\cite{Zhang2015}). 

Stochastic orders and related inequalities play an important role in
reliability theory and maintenance policies, as they allow to, e.g., obtain
bounds for system reliability or availability, or to compare different
maintenance strategies (\cite{barlow1964}, \cite{Ohnishi2002}). There is a huge reliability literature on the use of
stochastic orders which compare locations of the lifetime, residual lifetime
or inactivity time of the systems (\cite{khaledi2007}). However, there exist
other types of stochastic orders which measure variability and spread.
Though their use has become classical in insurance literature (\cite
{denuit1997} and \cite{denuit1999} {\black amongst others}), they are
not so common in the reliability literature apart from a few exceptions (\cite
{kochar2009}, \cite{fang2013}).

Following the spirit showed in \cite{MercierCastro2013} and \cite{Castro2016} {\black (see also \cite{GiorgioPulcini2018})}, two
models of imperfect repair are analyzed in this paper {\black for a gamma deteriorating system}. The first model, called Arithmetic Reduction of {\black of Deterioration} of order 1 (ARD1), assumes that the repair removes the $\rho_1\%$ of the degradation accumulated by the system from the last maintenance action. The second model {\black is based on the notion of virtual age as introduced by \cite{Kijima1989} in the context of recurrent events (where only lifetime data are available). The idea is that an imperfect repair rejuvenates the system, namely puts it back to a similar state as it was before the repair (details further). Following \cite{doyen2004} in the context of recurrent events, an Arithmetic Reduction of Age of order 1 (ARA1) is here considered,} which assumes that 
the repair removes the $\rho_2\%$ of the (virtual) age accumulated by the system since the last maintenance action. An ARD1 repair hence lowers the deterioration level, without rejuvenating the system. On the contrary, by an ARA1 repair, the system is put back to the exact situation where it was some time before, which entails the lowering of both its deterioration level and (virtual) age. The two models may hence correspond to different maintenance actions in an {\color{black}application context.  As an example, \cite{Sadeghi2018} show that the tamping or cleaning of the ballast of a railway track may have different consequences: The tamping mainly improves the track geometry conditions (short term impact) but has mostly no impact on the ballast mechanical conditions (long term impact), whereas the cleaning also improves the latter. Then, one could think that the tamping corresponds to some reduction of the track deterioration level (such as an ARD1 repair) whereas the cleaning also acts on its potential of future degradation, which could be represented by some reduction of age model (such as an ARA1 repair).} 
 
{\color{black} The choice between the two models may however not always be so clear in an  applied context. For a better understanding of their differences,} this paper focuses on their comparison, from a probabilistic point of view. Assuming that the degradation of the system is modeled by a non homogeneous gamma process, stochastic comparisons of both location and spread of the two resulting processes are given. Moreover, a specific case is analyzed, where the two models provide identical expected deterioration levels at repair times (``equivalent'' case). 
 
 Going back to the general setup, two maintenance strategies are next proposed. Both strategies consider periodic imperfect repairs (period $T$) based on either one of the two models (ARD1 or ARA1). In the first strategy ($(n,T)$ policy), the system is replaced at the time of the $n$-th repair. The second strategy ($(M,T)$ policy) considers a control limit rule, with replacement when the degradation level exceeds a preventive threshold $M$. For both maintenance strategies, the objective  
function is the expected profit rate, which takes into account some reward produced by the system, with unitary reward (or cost) per unit time depending on the degradation level of the system. (The lower the deterioration level, the higher the unitary reward per unit time). This reward function is based on
classical utility functions used in insurance literature (\cite{rolski1998}). The
use of this reward function represents an advance in the reliability
literature where the cost objective function is usually developed considering that the system state is binary (up or down), with some fixed unavailability cost per unit time, independent on the deterioration level. Theoretical results are obtained for the comparison of the objective functions of the $(n,T)$ policy under the two types of imperfect repairs (ARD1 or ARA1). 

The paper is organized as follows: Section 2 provides some technical reminders. The two imperfect repair models are described in Section 3. The corresponding moments are compared in Section 4 whereas Section 5 is devoted to stochastic comparison results. The ``equivalent'' case is studied in Section 6.  Section 7 deals with the reward function and the two maintenance strategies. Concluding remarks are provided in Section 8, together with possible extensions.

\section{Technical reminders}

The definition of two stochastic orders is first recalled, which allows to
compare the location of random variables.

\begin{definition}
\label{Def Sto Order} Let $X$ and $Y$ be two non negative random variables
with probability density functions (p.d.f.) $f_{X}$ and $f_{Y}$ with respect
to the Lebesgue measure, cumulative distribution functions (c.d.f.) $F_{X}$ and $%
F_{Y}$ and survival functions $\bar{F}_{X}$ and $\bar{F}_{Y}$, respectively.
Then:

\begin{enumerate}
\item $X$ is said to be smaller than $Y$ in the usual stochastic order ($%
X\prec_{sto}Y$) if $\bar{F}_{X}\leq\bar{F}_{Y}$ (or $F_{X}\geq F_{Y}$,
equivalently).

\item $X$ is said to be smaller than $Y$ in the likelihood ratio order ($%
X\prec_{lr}Y$) if $\frac{f_{Y}}{f_{X}}$ is non-decreasing on the union of
the supports of $X$ and $Y$.
\end{enumerate}
\end{definition}

We recall that the likelihood ratio order implies the usual stochastic
order. The definition of other stochastic orders is next provided, which
allows to compare the variability of two random variables.

\begin{definition}
\label{Def Var Order}Let $X$ and $Y$ be two non negative random variables
where the support of $X$ is assumed to be included in the support of $Y$ and
the support of $Y$ to be an interval (for the log-concavity). Then:

\begin{enumerate}
\item $X$ is said to be smaller than $Y$ in the log-concave order ($%
X\prec_{lc}Y$) if the ratio $\frac{f_{X}}{f_{Y}}$ is log-concave over the
support of $Y$.

\item $X$ is said to be smaller than $Y$ in the convex (concave) order ($%
X\prec_{cx\left( cv\right) }Y$) if $\mathbb{E}\left( \varphi\left( X\right)
\right) \leq\mathbb{E}\left( \varphi\left( Y\right) \right) $ for all convex
functions $\varphi$ (provided the expectations exist).

\item $X$ is said to be smaller than $Y$ in the increasing convex (concave)
order ($X\prec_{icx\left( icv\right) }Y$) if $\mathbb{E}\left( \varphi\left(
X\right) \right) \leq\mathbb{E}\left( \varphi\left( Y\right) \right) $ for
all increasing convex (concave) functions $\varphi$ (provided the
expectations exist).
\end{enumerate}
\end{definition}

Following \cite{shaked2007}, $X\prec_{icx}(\prec_{icv}) Y$ roughly means
that $\mathbb{E}(X)\leq\mathbb{E}(Y)$ (location condition) plus the fact
that $X$ is less (more) ``variable'' than $Y$, in a stochastic sense. Also, $%
X\prec_{cx}(\prec_{cv})Y$ is equivalent to $X\prec_{icx}(\prec_{icv})Y$ plus 
$\mathbb{E}(X)=\mathbb{E}(Y)$. 

Setting $X$ and $Y$ to be two non negative random variables, we recall
	\cite[page 182]{shaked2007} that $X\prec_{icx}Y$ if and only if
	\begin{equation}\label{icx}
	\int_{x}^{+\infty}\bar{F}_{X}\left(  u\right)  ~du\leq\int_{x}^{+\infty}%
	\bar{F}_{Y}\left(  u\right)  ~du\text{ for all }x\geq0
	\end{equation}
	and that $X\prec_{icv}Y$ if and only if
	\begin{equation}\label{icv}
	\int_{0}^{x}F_{X}\left(  u\right)  ~du\geq\int_{0}^{x}F_{Y}\left(  u\right)
	~du\text{ for all }x\geq0.
	\end{equation}

Finally, the usual stochastic order (and
hence the likelihood ratio order as well) implies both increasing convex and
concave orders \cite[p.61]{muller2002}.

We next come to reminder on gamma distribution. Let $a,b>0$. We recall that
the gamma distribution $\Gamma \left( a,b\right) $ with parameters $\left(
a,b\right) $ admits 
\begin{equation*}
f_{a,b}\left( x\right) =\frac{b^{a}}{\Gamma \left( a\right) }x^{a-1}e^{-bx}%
\mathbf{1}_{\mathbb{R}_{+}}\left( x\right)
\end{equation*}%
as p.d.f. (with respect to Lebesgue measure) and that the corresponding mean
and variance are $\frac{a}{b}$ and $\frac{a}{b^{2}}$, respectively. We shall
also make use of the following well-known facts repeatedly, without further
notification: If $X$ is gamma distributed $\Gamma \left( a,b\right) $, then $%
c~X$ is gamma distributed $\Gamma \left( a,b/c\right) $ for all $c>0$. If $X_{1},\cdots ,X_{n}$ are independent gamma distributed random variables
with respective distributions $\Gamma \left( a_{1},b\right) ,\cdots ,\Gamma
\left( a_{n},b\right) $, then $\sum_{i=1}^{n}X_{i}$ is gamma distributed $%
\Gamma \left( \sum_{i=1}^{n}a_{i},b\right) $.

Finally, the following technical result may be found in \cite[p. 62]%
{muller2002}.

\begin{lemma}
\label{Lem gam} Let $X$ and $Y$ be gamma distributed random variables  
with parameters $\left( a_{1},b_{1}\right) $ and $\left(
a_{2},b_{2}\right) $, respectively, where $a_{i},b_{i}>0$ for $i=1,2$. Then:

\begin{enumerate}
\item If $a_{1}\leq a_{2}$ and $b_{1}\geq b_{2}$, then $X\prec_{lr}Y$;

\item If $a_{1}\geq a_{2}$ and $a_{1}/b_{1}\leq a_{2}/b_{2}$, then $%
X\prec_{icx}Y$;

\item If $a_{1}\leq a_{2}$, $b_{1}\leq b_{2}$ and $a_{1}/b_{1}\leq
a_{2}/b_{2}$, then $X\prec_{icv}Y$.
\end{enumerate}
\end{lemma}

\section{The two models of imperfect repairs}

In the sequel of this work, we set $\left( X_{t}\right) _{t\geq0}$ to be the
{\black intrinsic (out of repair)} degradation process of the system. We assume that $\left(
X_{t}\right) _{t\geq0}$ follows a non homogeneous gamma process with
parameters $A(\cdot)$ and $b$, where $A(\cdot):\mathbb{R}_{+}\longrightarrow 
\mathbb{R}_{+}$ is continuous and non-decreasing with $A\left( 0\right) =0$,
and $b>0$. We recall that $\left( X_{t}\right) _{t\geq0}$ is a process with
independent increments such that $X_{0}=0$ almost surely (a.s.) and such that each increment $X_{t+s}-X_{t}$ is gamma
distributed $\Gamma(A(t+s)-A(t),b)$ for all $s,t>0$.

The system is periodically and instantaneously maintained each $T$ units of
time. For modeling purpose, we set $X^{\left( i\right) },i\in \mathbb{N}%
^{\ast }$ to be i.i.d. copies of $X=\left( X_{t}\right) _{t\geq 0}$, where $%
X^{\left( i\right) }$ describes the evolution of the deterioration level
between the $i$-th and $(i+1)$-th maintenance actions. For each imperfect
repair model, the maintenance {\black efficiency} is measured by an Euclidian parameter $%
\rho \in (0,1)$.

\subsection{First model: Arithmetic Reduction of Deterioration of order 1
(ARD1)}

In this model, the maintenance action instantaneously removes the $\rho\%$ of
the degradation accumulated by the system from the last
maintenance action (or from the origin). Let $\left( Y_{t}\right) _{t\geq0}$
be the process that describes the degradation level of the maintained system
under this model of repair.

The ARD1 model is developed as follows: At the beginning, the system
deteriorates according to $X^{(1)}$ and it is first maintained at time $T$.
This provides: 
\begin{equation*}
Y_{t}=X_{t}^{\left( 1\right) }\quad t<T,\quad \quad Y_{T}=\left( 1-\rho
\right) X_{T}^{\left( 1\right) }.
\end{equation*}%
Between $T$ and $2T$, the system deteriorates according to $X^{(2)}$. The
age of the system is unchanged at time $T$ and we simply have 
\begin{equation*}
Y_{t}=Y_{T}+\left( X_{t}^{\left( 2\right) }-X_{T}^{\left( 2\right) }\right)
\end{equation*}%
for all $T\leq t<2T$, and at the second maintenance time $2T$: 
\begin{equation*}
Y_{2T}=Y_{T}+\left( 1-\rho \right) \left( X_{2T}^{\left( 2\right)
}-X_{T}^{\left( 2\right) }\right) .
\end{equation*}%
More generally, we get:%
\begin{eqnarray} 
Y_{t} &=&Y_{nT}+\left( X_{t}^{\left( n+1\right) }-X_{nT}^{\left( n+1\right)
}\right) \text{ for all }nT\leq t<\left( n+1\right) T, \\ \label{Yt}
Y_{\left( n+1\right) T} &=&Y_{nT}+\left( 1-\rho \right) \left( X_{\left(
n+1\right) T}^{\left( n+1\right) }-X_{nT}^{\left( n+1\right) }\right) \label{YDiffT}
\end{eqnarray}%
where $X_{t}^{\left( n+1\right) }-X_{nT}^{\left( n+1\right) }$ is gamma
distributed $\Gamma (A\left( t\right) -A\left( nT\right) ,b)$ for all $t\in %
\left[ nT,\left( n+1\right) T\right] $. Hence %
\begin{equation} \label{YT}
Y_{nT}=\left( 1-\rho \right) \sum_{i=1}^{n}\left( X_{iT}^{\left( i\right)
}-X_{\left( i-1\right) T}^{\left( i\right) }\right)
\end{equation}%
is gamma distributed $\Gamma \left( A\left(
nT\right) ,\frac{b}{1-\rho }\right)$. 
%\begin{equation*}
%\Gamma \left( \sum_{i=1}^{n}\left( A\left( iT\right) -A\left( \left(
%i-1\right) T\right) \right) ,\frac{b}{1-\rho }\right) =\Gamma \left( A\left(
%nT\right) ,\frac{b}{1-\rho }\right) .
%\end{equation*}

Except for the case $\rho \rightarrow 0^{+}$, if $t \, \mbox{mod} \, T \neq
0 $ , $Y_{t}$ is the sum of two independent and gamma distributed random
variables (r.v.s) with different scale parameters, and it is not gamma
distributed. Its expectation and variance are given by:
\begin{eqnarray} \label{EYt}
\mathbb{E}\left( Y_{t}\right) =\frac{A\left(
t\right) -\rho ~A\left( nT\right) }{b},  \quad 
var\left( Y_{t}\right) =  \frac{A\left( t\right) -\rho \left( 2-\rho \right) A\left( nT\right) }{%
b^{2}}, 
\end{eqnarray}
for $nT\leq t<\left( n+1\right) T$.

\begin{remark}
\label{Rem EV Y}It is easy to check that $\mathbb{E}\left( Y_{t}\right) $
and $var\left( Y_{t}\right) $ are decreasing with respect to $\rho $.
\end{remark}

Next result provides some more insight than the previous remark into the
impact of the maintenance efficiency $\rho $ on the deterioration level of the
maintained system. To state it, two different efficiency parameters $\rho
_{1}$ and $\rho _{2}$ ($\rho _{1},\rho _{2}\in (0,1)$) are envisioned. The
resulting ARD1 processes are denoted by $\left( Y_{t}^{\left( i\right)
}\right) _{t\geq 0}$ for $i=1,2$, respectively.

\begin{proposition}
\label{Prop Y rho}We have:

\begin{enumerate}
\item $Y_{nT}$ decreases with respect to $\rho $ in the sense of the likelihood order: If $%
\rho_{1}<\rho_{2}$, then $Y_{nT}^{\left( 2\right) }\prec_{lr}Y_{nT}^{\left(
1\right) }$;

\item $Y_{t}$ decreases with respect to $\rho $ in the sense of both increasing convex and concave
orders: If $\rho_{1}<\rho_{2}$, then $Y_{t}^{\left( 2\right)
}\prec_{icx}Y_{t}^{\left( 1\right) }$ and $Y_{t}^{\left( 2\right)
}\prec_{icv}Y_{t}^{\left( 1\right) }$.
\end{enumerate}
\end{proposition}

\begin{proof}
From (\ref{YT}), we know that $Y_{nT}^{\left( i\right) }\sim \Gamma \left( A\left( nT\right) ,%
\frac{b}{1-\rho _{i}}\right) $ for $i=1,2$. If $\rho _{1}<\rho _{2}$, we
have $\frac{b}{1-\rho _{1}}\leq \frac{b}{1-\rho _{2}}$ and the first result
follows from point 1 of Lemma \ref{Lem gam}, which entails that $%
Y_{nT}^{\left( 2\right) }\prec _{icx}Y_{nT}^{\left( 1\right) }$ and $%
Y_{nT}^{\left( 2\right) }\prec _{icv}Y_{nT}^{\left( 1\right) }$.

{\color{black} Let $nT\leq t<\left( n+1\right) T$. Based on (\ref{Yt}), we have
\begin{equation*}
Y_{t}^{\left( 2\right) }-Y_{nT}^{\left( 2\right) }=Y_{t}^{\left( 1\right)
}-Y_{nT}^{\left( 1\right) }=X_{t}^{(n+1)}-X_{nT}^{(n+1)}
\end{equation*}%
so that $Y_{t}^{\left( 2\right) }-Y_{nT}^{\left( 2\right) } \prec
_{icx}Y_{t}^{\left( 1\right) }-Y_{nT}^{\left( 1\right) }$.
We derive that} $$Y_{t}^{\left( 2\right) }=Y_{nT}^{\left( 2\right) }+\left(
Y_{t}^{\left( 2\right) }-Y_{nT}^{\left( 2\right) }\right) \prec
_{icx}Y_{t}^{\left( 1\right) }=Y_{nT}^{\left( 1\right) }+\left(
Y_{t}^{\left( 1\right) }-Y_{nT}^{\left( 1\right) }\right) $$ because the icx
order is stable through convolution \cite[Thm 4.A.8 (d) page 186]%
{shaked2007}. The proof is similar for $\prec _{icv}$.
\end{proof}

\begin{remark}\label{Rem Y}
Note that $\prec_{lr}$ is not stable under convolution in a general setting
so that the second point of the previous proposition would not be valid for
this order (except if $X^{(n+1)}_{t}-X^{(n+1)}_{nT}$ has a log-concave
density \cite[Thm 1.C.9. page 46]{shaked2007}, namely if $A\left( t\right)
-A\left( nT\right) \geq1$ for $t \in [nT,(n+1)T)$). {\black A counter-example showing that the second point does not hold for the likelihood order is provided in Remark \ref{Rem 4} later on.}
\end{remark}

\subsection{Second model: Arithmetic Reduction of (virtual) Age of order 1
(ARA1)}
{\black The ARA1 model is based on the notion of virtual age as introduced by \cite{Kijima1989} in the context of recurrent events, which we first recall. Let us consider a system with successive lifetimes $U_1$, $U_2$, \dots, $U_n$, \dots and instantaneous repairs at times $T_n=\sum_{i=1}^{n}U_i,n=1,2,\dots$ (with $T_0=0$). Let $\bar{F}$ be the survival function of $U_1$. Assume that there exists a sequence $(V_n)_{n\geq 1}$ of non negative random variables such that after the $n-th$ maintenance action at time $T_n$, the next lifetime $U_{n+1}$ has the same conditional distribution given $V_n=y$ as the remaining lifetime of a new system at time $y$:
	\begin{equation}
		\mathbb{P}(U_{n+1}>t|V_n=y)=\mathbb{P}(U_{1}-y>t|U_{1}>y)=\frac{\bar{F}(t+y)}{\bar{F}(y)}\label{virtual age}
	\end{equation}
for  all $t,y>0$.
	Then, $V_n$ is called the virtual age of the system at time $T_n$. After a maintenance action at time $T_n$, the next lifetime $U_{n+1}$ has the same distribution as if the calendar age of the system were equal to $V_n$. Between repairs, the virtual age evolves with speed 1, just as the calendar time, so that for $T_n\leq t <T_{n+1}$, the virtual age is $V(t)=V_n+(t-T_n)$. In the context of the present paper, we use a similar notion of virtual age for deteriorating systems. To be more specific, considering a system with intrinsic  deterioration modeled by $(X_t)_{t\geq 0}$ and instantaneous repairs at times $T_n,n=1,2,\dots$ (with $T_0=0$), we say that $V(t)=V_n+(t-T_n)$ stands for the virtual age of the system at time $t \in [T_n,T_{n+1})$ if, given $V_n$, the deterioration level of the maintained system is conditionally identically distributed as $(X_{V_n+t-T_n})_{T_n\leq t<T_{n+1}}$ (with independence between the deterioration and the virtual age, details further). The idea is just the same as for recurrent events: After a maintenance action at time $T_n$, the system behaves just as if its calendar age were equal to $V_n$ at time $T_n$ and between maintenance actions, the virtual age evolves with speed 1.
	
	We now come to the specific virtual age model developed in this paper, which is called Arithmetic Reduction of Age of order 1 (ARA1) model after \cite{doyen2004}. It is based on the Kijima II imperfect repair model \cite{Kijima1989} and each (periodic) repair removes the $\rho\%$ of the age accumulated by the system since the last maintenance action (or from the origin).}

Let $\left( Z_{t}\right) _{t\geq 0}$ be the process that describes the
degradation level of the maintained system under this model of repair. At
the first maintenance time $T$, the virtual age of the system is reduced by $\rho T$ units and it becomes $V(T)=(1-\rho )T$. {\black Recalling that $(X_t^{(i)})_{t \geq 0}$ models the deterioration between the $i$-th and $i+1$-th repairs,} the degradation level on $[0,T]$ is given by 
\begin{equation*}
Z_{t}=X_{t}^{\left( 1\right) }\text{ for }t<T,\quad \quad Z_{T}=X_{\left(
1-\rho \right) T}^{\left( 1\right) }.
\end{equation*}%
{\color{black} It
means that, at time $T$, the system goes back into its past: 
 The system is rejuvenated (from the age $T$ to the age $(1-\rho)T)$ and the deterioration level is reduced (from $X_T^{(1)}$ to $X_{(1-\rho)T}^{(1)}$). In case of $A(\cdot)$ convex, the deterioration rate is also reduced (from $A'(T)$ to $A'((1-\rho)T)$).

For $T\leq t<2T$, the system age is $V(t)=V(T)+(t-T)=t-\rho T$. The corresponding deterioration level is identically distributed as $X_{t-\rho T}$. It is equal to the deterioration level at time $T$ plus the increment of deterioration on $(T,t]$, which leads to
\begin{equation*}
Z_{t}=Z_{T}+\left( X_{t-\rho T}^{\left( 2\right) }-X_{\left( 1-\rho \right)
T}^{\left( 2\right) }\right),
\end{equation*}%
where $X_{t-\rho T}^{\left( 2\right) }-X_{\left( 1-\rho \right)
	T}^{\left( 2\right) }$ is independent on $Z_T=X_{\left(
	1-\rho \right) T}^{\left( 1\right)}$.}
At time $2T^{-}$ (just before the repair), the age of the system is $V(2T^{-})=2T-\rho T$ which is reduced by $\rho T$ at time $2T$. The age hence is $V(2T)=2(1-\rho )T$ at time $2T$. The corresponding deterioration level is given by:
\begin{equation*}
Z_{2T}=Z_{T}+\left( X_{2\left( 1-\rho \right) T}^{\left( 2\right)
}-X_{\left( 1-\rho \right) T}^{\left( 2\right) }\right) .
\end{equation*}%
More generally, for $nT\leq t<\left( n+1\right) T$, the virtual age at time $t$ is $V(t)=t-\rho nT$ (just as for an ARA1 model for recurrent events, see \cite{doyen2004})) and the system degradation is given by
\begin{eqnarray} 
Z_{t}& =& Z_{nT}+\left( X_{t-\rho nT}^{\left( n+1\right) }-X_{\left( 1-\rho
\right) nT}^{\left( n+1\right) }\right), \\  \label{Zt}
Z_{(n+1)T}& =& Z_{nT}+\left( X_{\left( 1-\rho \right)
(n+1)T}^{(n+1)}-X_{\left( 1-\rho \right) nT}^{(n+1)}\right) \label{ZDifft}
\end{eqnarray}%
where $X_{t-\rho nT}^{\left( n+1\right) }-X_{\left( 1-\rho \right)
nT}^{\left( n+1\right) }$ is gamma distributed $\Gamma \left( A\left( t-\rho
nT\right) -A\left( \left( 1-\rho \right) nT\right) ,b\right) $ for all $t\in %
\left[ nT,\left( n+1\right) T\right] $. 

Hence 
\begin{equation}\label{ZT}
Z_{nT}=\sum_{i=1}^{n}\left( X_{\left( 1-\rho \right) iT}^{\left( i\right)
}-X_{\left( 1-\rho \right) \left( i-1\right) T}^{\left( i\right) }\right)
\end{equation}%
and it is gamma distributed $\Gamma (A\left( (1-\rho)nT \right),b)$. Here, $%
Z_{t}$ is the sum of two independent gamma distributed r.v.s which share the
same scale parameter $b$ and it is gamma distributed $\Gamma (A(t-\rho
nT),b) $. 

Also: 
\begin{eqnarray} 
\mathbb{E}\left( Z_{t}\right) =\frac{A\left( t-\rho nT\right) }{b}, \quad 
var\left( Z_{t}\right) =\frac{A\left( t-\rho nT\right) }{b^{2}} \label{EZt}
\end{eqnarray}%
for all $nT\leq t<(n+1)T$.

\begin{remark}
\label{Rem EV Z}Here again, $\mathbb{E}\left( Z_{t}\right) $ and $var\left(
Z_{t}\right) $ are decreasing with respect to~$\rho $.
\end{remark}

Just as for the ARD1 model, we next set $\left( Z_{t}^{\left( i\right)
}\right) _{t\geq0}$ to be the ARA1 process with repair efficiency $\rho_i$
for $i=1, 2$.

\begin{proposition}
\label{Prop Z rho}$Z_{t}$ decreases with respect to $\rho$ for the
likelihood ratio order (and hence also for both increasing convex and
concave orders): If $\rho_1 < \rho_2$, then $Z_{t}^{\left( 2\right)
}\prec_{lr}Z_{t}^{\left( 1\right) }$.
\end{proposition}

\begin{proof}
Let $\rho_{1}<\rho_{2}$ and $nT\leq t<\left( n+1\right) T$. For $i=1,2$, from (\ref{Zt}), we have:  
$$
Z_{t}^{\left( i\right) }\sim\Gamma\left( A\left( t-\rho_{i}nT\right)
,b\right), $$ with $A\left( t-\rho_{2}nT\right) \leq A\left(
t-\rho_{1}nT\right) $. We hence derive from Lemma \ref{Lem gam} that $%
Z_{t}^{\left( 2\right) }\prec_{lr}Z_{t}^{\left( 1\right) }$.
\end{proof}
{\black 
\begin{remark}\label{Rem 4}
	We can see from Propositions \ref{Prop Y rho} and \ref{Prop Z rho} that as expected, the more efficient the maintenance action is (namely the larger $\rho$ is), the smaller the deterioration level is for both ARD1 and ARA1 models. Based on \cite[Thm 1.C.5. p 44]{shaked2007}, the previous result implies for instance that, for an ARA1 model, we have 
	\begin{equation}\label{rem 4}
		[Z_t^{(2)}|Z_t^{(2)}>h]\leq_{sto}[Z_t^{(1)}|Z_t^{(1)}>h]
	\end{equation}
	for all $h>0$. Imagine that $h$ is an alert threshold in an application context and that the crossing of $h$ triggers a signal, then, the previous relation means that, given that the signal has already been triggered, the deterioration level is stochastically all the smaller as the efficiency of the maintenance action is higher. Now considering  $A(t)=t^{2},b=1,h=0.5,t=1.9,\rho_1=0.9<\rho_2=0.95$, we get that $\mathbb{E}(Y_t^{(2)}|Y_t^{(2)}>h)\simeq 1.16 > \mathbb{E}(Y_t^{(1)}|Y_t^{(1)}>h)\simeq 1.08$, which shows that (\ref{rem 4}) is not valid any more for the ARD1 model, in concordance with 
	Remark \ref{Rem Y}. The stronger likelihood ratio result obtained for the ARA1 model may hence lead to different consequences from those for the ARD1 model in an application context (and in particular, conditional expectations are not necessarily ranked in an intuitive way).
\end{remark}
}
\section{Comparison of the moments}

We now come to the main object of the paper, which is the comparison between
the two models of imperfect repairs. Note that, in an application context,
there is no reason why the estimated repair efficiency should be the same
when the impact of the maintenance is modeled by an ARD1 or ARA1 model. Our point hence is to compare $%
Y_{t}^{(1)}$ and $Z_{t}^{(2)}$, with efficiency $\rho _{1}$ and $\rho _{2}$,
respectively. In this section, we focus on the comparison of their respective means and variances.

\begin{proposition}
\label{means}Let us consider the two following assertions:
\begin{eqnarray}
\mathbb{E}\left( Z_{t}^{(2)}\right) &\geq& \mathbb{E}\left( Y_{t}^{(1)}\right) 
\text{ for all }t,T>0  \label{eq4}\\
\text{and } A\left( (1-\rho _{2}) t\right) &\geq& (1-\rho _{1})A(t)\text{ for all }t>0.
\label{eqA1}
\end{eqnarray}
Then:
\begin{enumerate}
	\item Assertion $(\ref{eq4})$ implies assertion $(\ref{eqA1})$.
	\item If $A(\cdot )$ is concave, then the converse is also true (namely $\left( \ref{eqA1}\right) \Longrightarrow
	\left( \ref{eq4}\right) $).
	\item As a special case, if $A(\cdot )$ is concave and $\rho _{2}\leq \rho _{1}$, then $\left( \ref{eq4}\right) $ is true.
\end{enumerate}
All the previous results are valid with reversed inequalities and concave
substituted by convex.
\end{proposition}

\begin{proof}
Based on $\left( \ref{EYt}\right) $ and $\left( \ref{EZt}\right) $,
assertion $\left( \ref{eq4}\right) $ is equivalent to 
\begin{equation}
A(t-\rho _{2}nT)\geq A(t)-\rho _{1}A(nT)  \label{eq5}
\end{equation}%
for all $T>0$, all $n\in \mathbb{N}$ and all $nT\leq t<\left( n+1\right) T$.
{\black Taking $t=nT$ in (\ref{eq5}), it implies $\left( \ref{eqA1}\right) $  for all $t=nN$ with $T>0$ and $n\in \mathbb{N}$, and hence for all $t>0$. This shows the first point.}

For the second point, assume $A(\cdot )$ to be concave. Then $%
A\left( t-\rho _{2}nT\right) -A(t)$ is non decreasing with respect to $t$
and for $nT\leq t<(n+1)T$, we get that 
\begin{equation}
A(t-\rho _{2}nT)-A(t)\geq A((1-\rho _{2})nT)-A(nT).  \label{eq1}
\end{equation}

Assuming $\left( \ref{eqA1}\right) $ to be true, we easily derive that 
\begin{equation*}
A(t-\rho _{2}nT)-A(t)\geq (1-\rho _{1})A(nT)-A(nT)=-\rho _{1}A\left(
nT\right)
\end{equation*}%
so that $\left( \ref{eq5}\right) $ is true. This implies $\left( \ref{eq4}%
\right) $ and the second point is proved.

In the specific case where $A(\cdot )$ is concave and $\rho _{2}\leq \rho
_{1}$, we have  
\begin{eqnarray}
A((1-\rho _{2})nT) =A\left( (1-\rho _{2})nT+\rho _{2}\:0\right)  
&\geq &(1-\rho _{2})A(nT)  \label{eq2} \\
&\geq &(1-\rho _{1})A(nT)  \notag
\end{eqnarray}%
for all $n$ and $T$ (using $A(0)=0$ in the first line). This implies that $%
\left( \ref{eqA1}\right) $ is true. Point three now is a direct consequence of point two.

The reasoning is similar for reversed inequalities and it is omitted.
\end{proof}

\begin{remark}
Note that Condition $\left( \ref{eqA1}\right) $ is less restrictive than $%
\rho _{2}\leq \rho _{1}$ in the concave case. (The same with a reverse
inequality in the convex case). For instance, consider $A\left( t\right)
=t^{\beta }$ with $0<\beta <1$. Then Condition $\left( \ref{eqA1}\right) $
means that $\left( 1-\rho _{2}\right) ^{\beta }\geq 1-\rho _{1}$, which is less restrictive than $\rho _{1}\geq \rho _{2}$.	
\end{remark}

In the following example, we look at the comparison of the expectations when the conditions in  Proposition \ref{means} are not fulfilled.
	\begin{example}\label{Ex Means} As a first case, we take $A\left(  t\right)  =e^{t}-1$ (convex function), $b=1$, $T=1$ $\rho_{1}=0.95$, $\rho_{2}=0.5$. As a second case, we take $A\left(  t\right)  =1-e^{-t}$ (concave function), $b=1$, $T=1$ $\rho_{1}=0.5$, $\rho_{2}=0.95$. Neither condition $\left( \ref{eqA1}\right) $ nor the reversed inequality is true on the whole real line. The corresponding expectations are plotted  in Figures \ref{Fig Mean1} (first case) and \ref{Fig Mean2} (second case). We can see that the respective means of $Y_{t}^{(1)}$ and $Z_{t}^{(2)}$ are not ordered in the same way on the whole real line.

		\begin{figure}[tbp]
			\begin{subfigure}{0.5\textwidth}
					\includegraphics[width=0.78\textwidth,angle =-90]{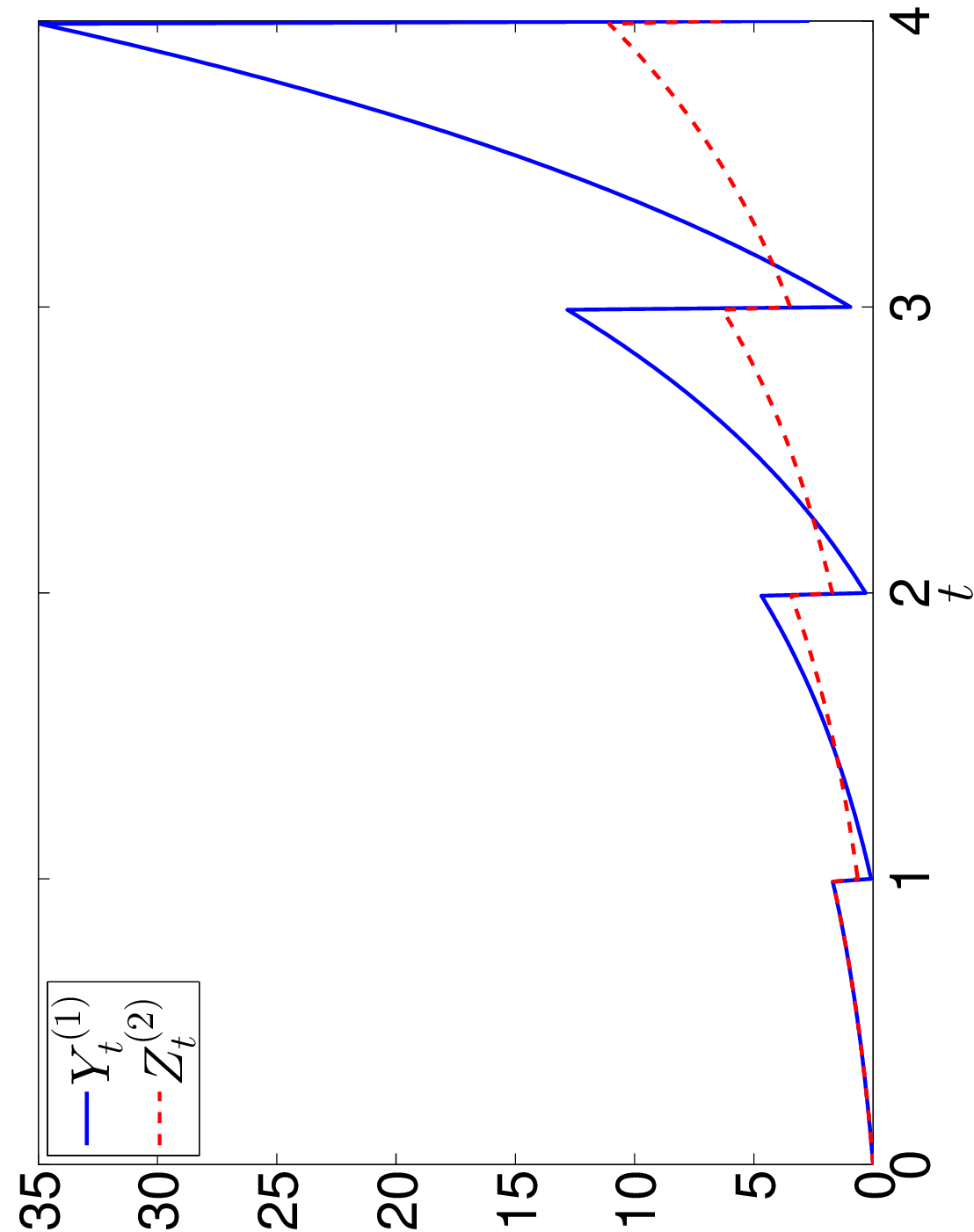}
				\caption{$A\left(  t\right)  =e^{t}-1$, $\rho_{1}=0.95$, $\rho_{2}=0.5$} \label{Fig Mean1}
			\end{subfigure}%
			\begin{subfigure} {0.5\textwidth}
				\includegraphics[width=0.78\textwidth,angle =-90]{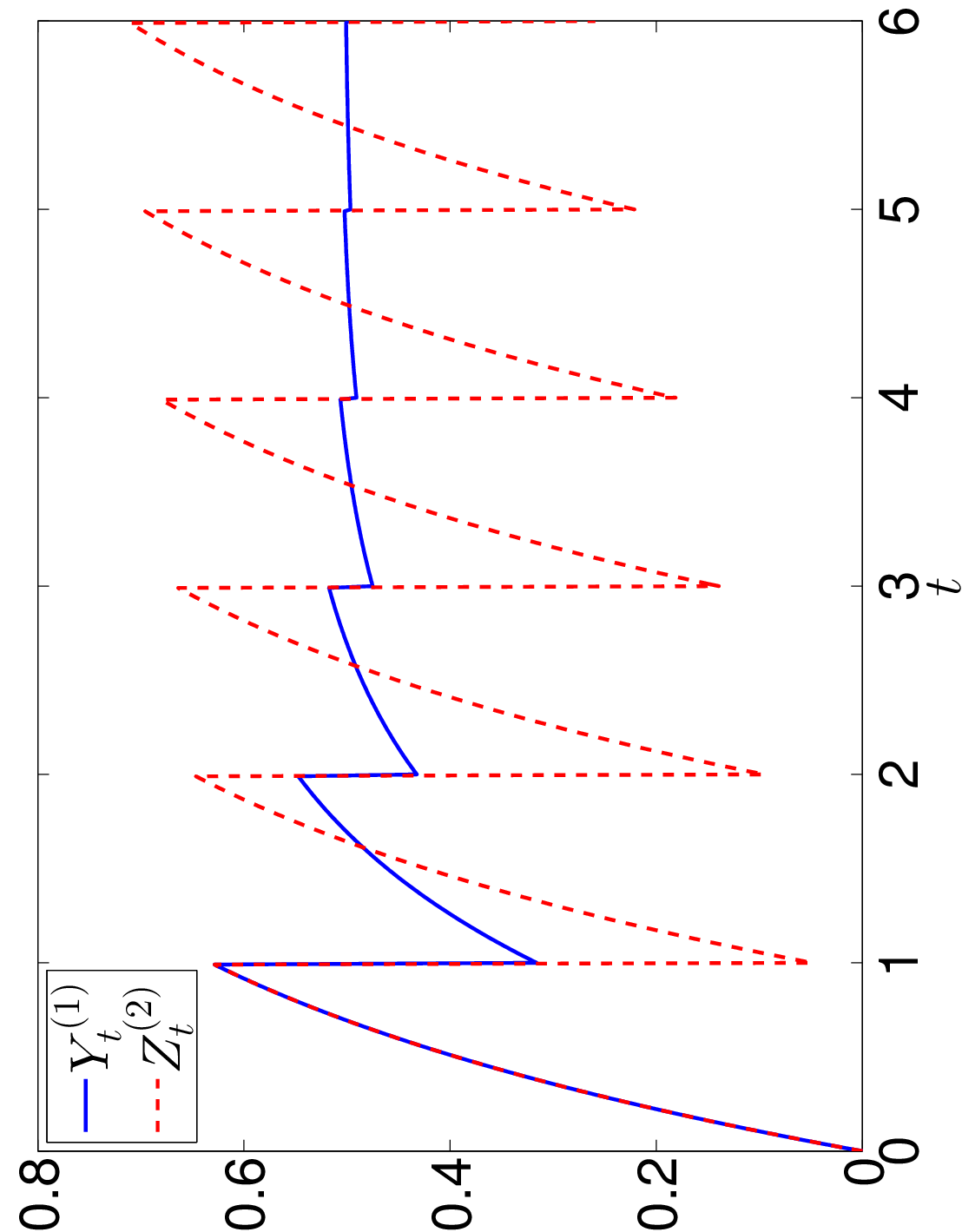}
				\caption{$A\left(  t\right)  =1-e^{-t}$,  $\rho_{1}=0.5$, $\rho_{2}=0.95$} \label{Fig Mean2}
			\end{subfigure}
			\caption{Expectations of $Y_{t}^{(1)}$ and $Z_{t}^{(2)}$ with $b=1$, $T=1$}
		\end{figure} 
	\end{example}

\begin{proposition}
\label{var}Let us consider the two following assertions:
\begin{eqnarray}
Var\left( Z_{t}^{(2)}\right) &\geq& Var\left( Y_{t}^{(1)}\right) \text{ for
	all }t,T>0  \label{V1}\\
\text{and }A\left( (1-\rho _{2})t\right) &\geq& (1-\rho _{1})^{2}A(t)\text{ for all }t>0.
\label{V2}
\end{eqnarray}
Then:
\begin{enumerate}
	\item Assertion $(\ref{V1})$ implies assertion $(\ref{V2})$.
	\item If $A(\cdot )$ is concave, then the converse is also true (namely $\left( \ref{V2}\right) \Longrightarrow
	\left( \ref{V1}\right) $). 
	\item As a special case, if $A(\cdot )$ is concave and $1-\rho _{2}\geq
	\left( 1-\rho _{1}\right) ^{2}$, then $\left( \ref{V1}\right) $ is true.
\end{enumerate}
All the previous results are valid with reversed inequalities, and concave
substituted by convex. 
\end{proposition}

\begin{proof}
Based on $\left( \ref{EYt}\right) $ and $\left( \ref{EZt}\right) $,
inequality $\left( \ref{V1}\right) $ is equivalent to 
\begin{equation}
A(t-\rho _{2}nT)\geq A(t)-\rho _{1}\left( 2-\rho _{1}\right) A(nT)
\label{V3}
\end{equation}%
for all $T>0$, all $n\in \mathbb{N}$, all $nT\leq t<\left( n+1\right) T$.
Considering $t=nT$, it implies $\left( \ref{V2}\right) $, so that point one is true.

Now assume $A(\cdot )$ to be concave. Based on $\left( \ref{eq1}\right) $, we
have:%
\begin{eqnarray*}
A(t-\rho _{2}nT)-A(t)+\rho _{1}(2-\rho _{1})A(nT) 
\geq A((1-\rho _{2})nT)-\left( 1-\rho _{1}\right) ^{2}A(nT).
\end{eqnarray*}

If $\left( \ref{V2}\right) $ is true, then $\left( \ref{V3}\right) $ is
consequently true and $\left( \ref{V1}\right) $ too. This shows the second point.

In the specific case where $A(\cdot )$ is concave and $(1-\rho _{2})\geq
\left( 1-\rho _{1}\right) ^{2}$, we have 
\begin{equation*}
A((1-\rho _{2})nT)\geq (1-\rho _{2})A(nT)\geq (1-\rho _{1})^{2}A(nT)
\end{equation*}%
(based on $\left( \ref{eq2}\right) $ for the first inequality). Hence $%
\left( \ref{V3}\right) $ is true, so that $\left( \ref{V1}\right) $ is true
too. This provides point three.

The reasoning is similar for reversed inequalities and it is omitted.
\end{proof}

\begin{example}
We now consider the same data as for Example \ref{Ex Means}, where neither condition $\left( \ref{V2}\right) $ nor the reversed inequality is true on the whole real line. The corresponding variances are plotted  in Figures \ref{Fig Var1} (first case) and \ref{Fig Var2} (second case). We can see that the respective variances of $Y_{t}^{(1)}$ and $Z_{t}^{(2)}$ are not ordered in the same way on the whole real line.

		\begin{figure}[tbp]
			\begin{subfigure}{0.5\textwidth}
				\includegraphics[width=0.78\textwidth,angle =-90]{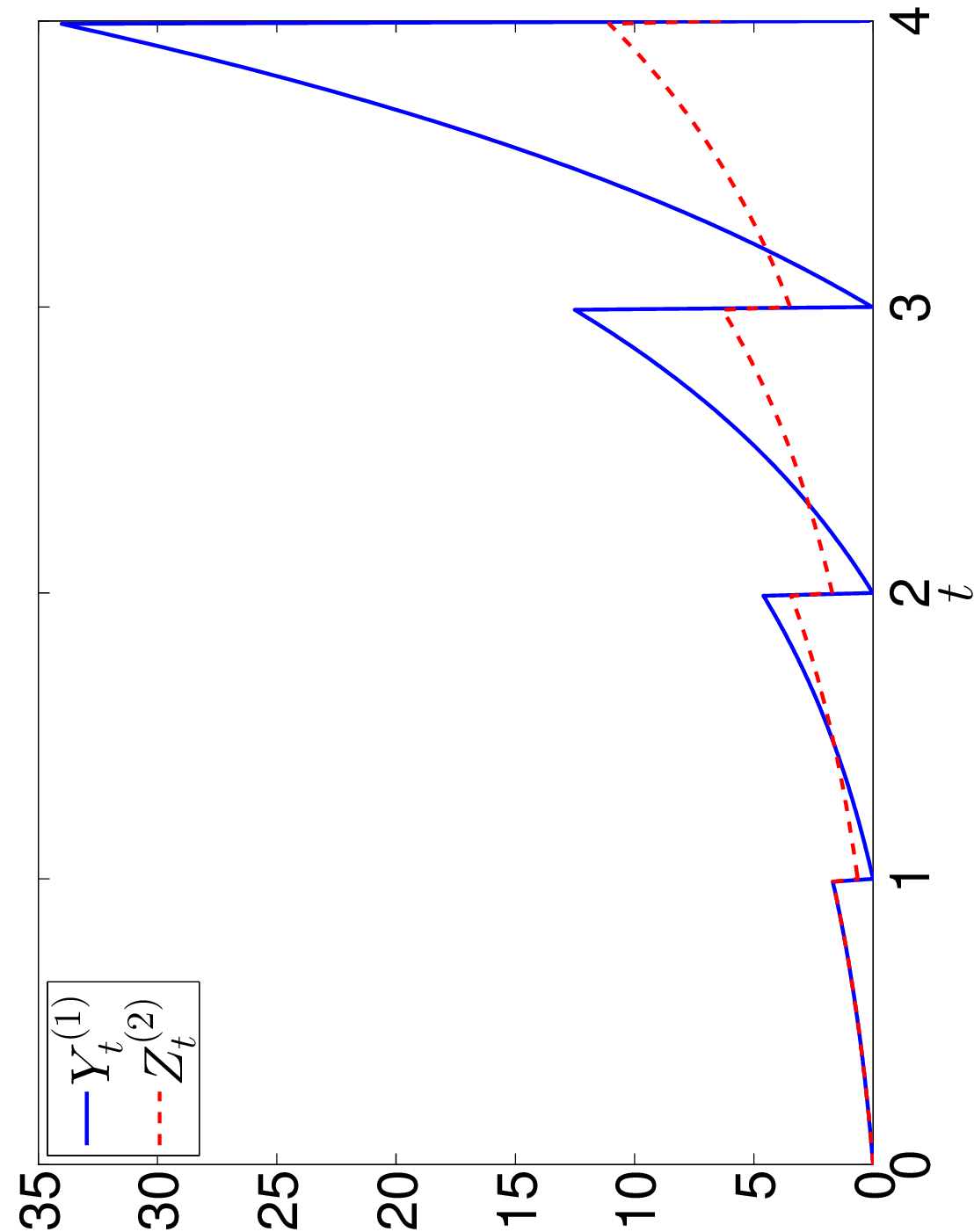}
				\caption{$A\left(  t\right)  =e^{t}-1$, $\rho_{1}=0.95$, $\rho_{2}=0.5$} \label{Fig Var1}
			\end{subfigure}%
			\begin{subfigure} {0.5\textwidth}
				\includegraphics[width=0.78\textwidth,angle =-90]{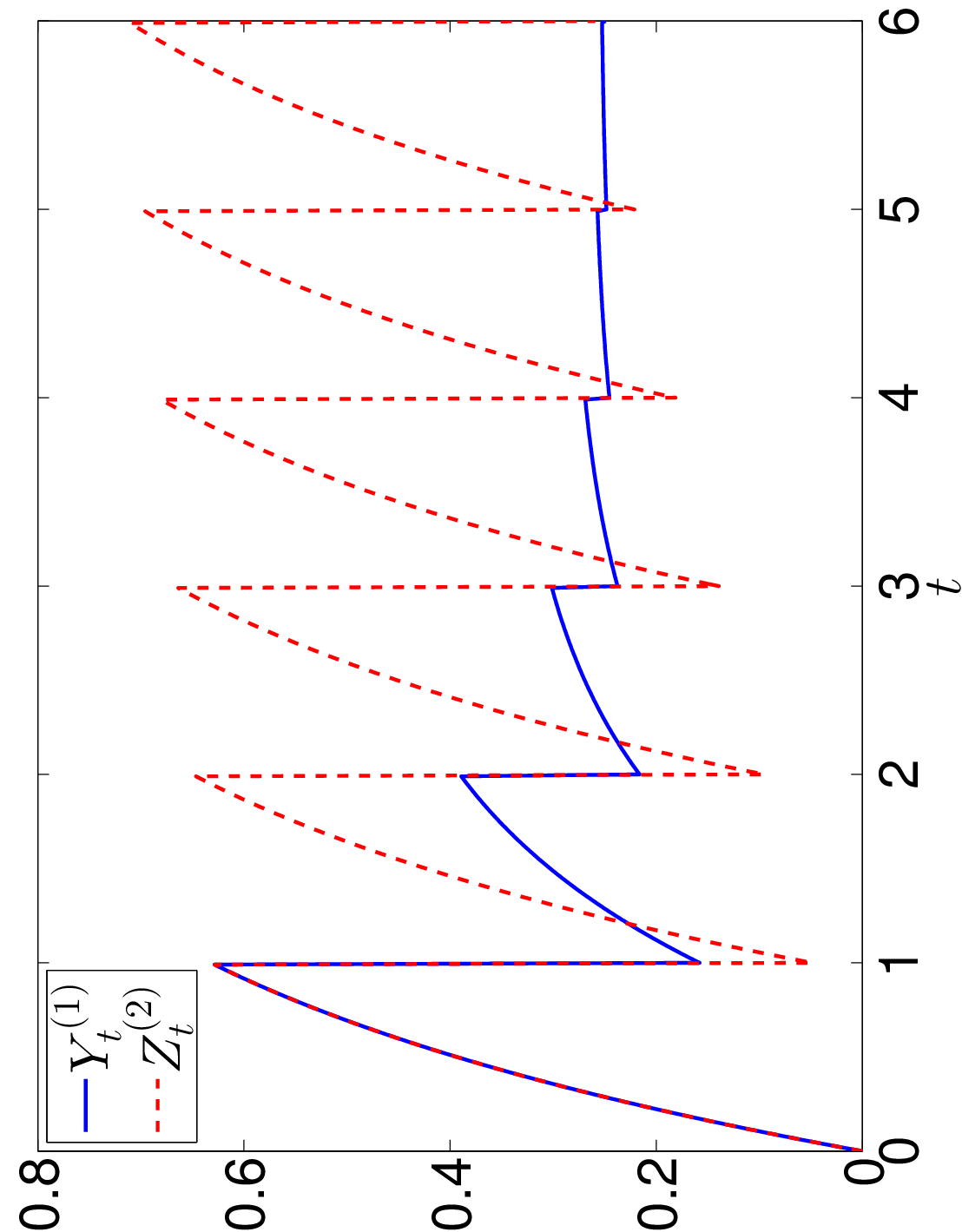}
				\caption{$A\left(  t\right)  =1-e^{-t}$,  $\rho_{1}=0.5$, $\rho_{2}=0.95$} \label{Fig Var2}
			\end{subfigure}
			\caption{Variances of $Y_{t}^{(1)}$ and $Z_{t}^{(2)}$ with $b=1$, $T=1$}
		\end{figure} 
\end{example}

Finally, we easily derive from Propositions \ref{means} and \ref{var} the following
corollary, where the expectation and variance of the two processes are
compared assuming a power law shape function for the gamma process.

\begin{corollary}
\label{corollaryweibullexpect}We consider $A(t)=\alpha t^{\beta }$ with $%
\alpha ,\beta >0$. Then we get

\begin{enumerate}
\item If $\beta \leq (\geq )1$, then 
$$\mathbb{E}\left( Z_{t}^{(2)}\right)
\geq \left( \leq \right) \mathbb{E}\left( Y_{t}^{(1)}\right) \forall
t\Leftrightarrow (1-\rho _{2})^{\beta }\geq \nolinebreak \left( \leq
\right) \left(1\nolinebreak -\nolinebreak \rho _{1}\right).$$
\item If $\beta \leq (\geq )1$, then  $$Var\left( Z_{t}^{(2)}\right) \geq
(\leq )Var\left( Y_{t}^{(1)}\right) \,\forall t  
\Leftrightarrow (1-\rho _{2})^{\beta }\geq \nolinebreak \left( \leq \right)
\left( 1-\rho _{1}\right) ^{2}.$$
\end{enumerate}
\end{corollary}

\section{Stochastic comparison between $Y_{t}^{(1)}$ and $Z_{t}^{(2)}$}

We now come to the stochastic comparison between $Y_t^{(1)}$ and $Z_t^{(2)}$, as given {\color{black} by (\ref{Yt}) and (\ref{Zt}), with $\rho$ substituted by $\rho_i,i=1,2$, respectively}.
\begin{proposition}
\label{Prop Icx Icv}If%
\begin{equation}
A\left( (1-\rho _{2})nT\right) \geq \left( \leq \right) (1-\rho _{1})A(nT),
\label{eqA2}
\end{equation}%
then $Y_{nT}^{(1)}\prec _{icx}\left( \succ _{icv}\right) Z_{nT}^{(2)}$.
\end{proposition}
\begin{proof}
	{\black From (\ref{YT}) and (\ref{ZT}), we know that $Y_{nT}^{(1)}$ and $Z_{nT}^{(2)}$ are gamma distributed, with distributions $\Gamma \left( A\left(
		nT\right) ,\frac{b}{1-\rho_1}\right)$ and $\Gamma (A\left( (1-\rho_2)nT \right),b)$, respectively. The results can hence be obtained by application of Lemma \ref{Lem gam}.}
\end{proof}
We next focus on the comparison {\black between} $Y_{t}^{(1)}-Y_{nT}^{(1)}$ and $Z_{t}^{(2)}-Z_{nT}^{(2)}$.
\begin{proposition}
\label{Prop_comp_incr}If $A\left( t\right) $ is convex (concave), then for
all $nT\leq t<\left( n+1\right) T$:

\begin{itemize}
\item $Z_{t}^{(2)}-Z_{nT}^{(2)}\prec _{lr}\left( \succ _{lr}\right)
Y_{t}^{(1)}-Y_{nT}^{(1)},$

\item $Y_{t}^{(1)}-Y_{nT}^{(1)}\prec _{lc}\left( \succ _{lc}\right)
Z_{t}^{(2)}-Z_{nT}^{(2)}.$
\end{itemize}
\end{proposition}

\begin{proof}
	{\black 
	For $nt\leq t<\left( n+1\right) T$, let us set   $$Y=Y_{t}^{(1)}-Y_{nT}^{(1)}=X_{t}^{\left( n+1\right)
	}-X_{nT}^{\left( n+1\right) }$$ (see (\ref{Yt})) and $$Z=Z_{t}^{(2)}-Z_{nT}^{(2)}=X_{t-\rho_{2}nT}^{\left( n+1\right) }-X_{\left( 1-\rho _{2}\right) nT}^{\left(
	n+1\right) }$$
(see (\ref{Zt})).
Then $Y$ and $Z$ are gamma distributed, with distributions $\Gamma(A(t)-A(nT))$ and $\Gamma(A\left( t-\rho _{2}nT\right) -A\left( \left( 1-\rho _{2}\right) nT\right))$, respectively. Now, the likelihood ratio comparison result is a direct consequence of Lemma \ref{Lem gam}, because the convexity (concavity) of $A(\cdot)$ entails that 
$$A(t)-A(nT) \geq (\leq) A\left( t-\rho _{2}nT\right)-A\left( \left( 1-\rho _{2}\right) nT\right). $$}

As for the log-concave order, using the notations of {\black Definition \ref{Def Var Order}}, we have%
\begin{equation*}
\left( \log \left( \frac{f_{Y}}{f_{Z}}\right) \right) ^{\prime }\left(
y\right) =\frac{\left[ A\left( t\right) -A\left( nT\right) \right] -\left[
A\left( t-\rho _{2}nT\right) -A\left( \left( 1-\rho _{2}\right) nT\right) %
\right] }{y}.
\end{equation*}%
This function decreases (increases) when $A\left( t\right) $ is convex
(concave), which provides the result.
\end{proof}

%The results of Proposition \ref{Prop_comp_incr} are illustrated in Figures \ref{likelihoodratio} and \ref{logcancava} at $t=10.5$ using $\rho_1=\rho_2=0.5$, $T=1$, $A(t)=t^{\alpha}$, $n=10$, $b=1$ and $\alpha=0.75$ (in the concave case) and $\alpha=1.25$ (in the convex case).
%
%
% 
%\begin{figure}[tbp]
%\begin{subfigure}{0.5\textwidth}
%\includegraphics[width=0.9\textwidth]{likelihoodratio.eps}
%\caption{$f_{Y_t^{(1)}-Y_{nT}^{(1)}}/f_{Z_t^{(2)}-Z_{nT}^{(2)}}$} \label{likelihoodratio}
%\end{subfigure}%
%\begin{subfigure} {0.5\textwidth}
%\includegraphics[width=0.9\textwidth]{logconcava.eps}
%\caption{$log\left(f_{Y_t^{(1)}-Y_{nT}^{(1)}}/f_{Z_t^{(2)}-Z_{nT}^{(2)}}\right)$} \label{logcancava}
%\end{subfigure}
%\caption{$f_{Y_t^{(1)}-Y_{nT}^{(1)}}/f_{Z_t^{(2)}-Z_{nT}^{(2)}}$ and $log\left(f_{Y_t^{(1)}-Y_{nT}^{(1)}}/f_{Z_t^{(2)}-Z_{nT}^{(2)}}\right)$ respectively}
%\end{figure}

\medskip

Hence, if $A\left( t\right) $ is convex (concave), the increment between
times $nT$ and $t$ (with $nT\leq t<\left( n+1\right) T$) is smaller (larger)
for the ARA1 model than for the ARD1 model in the sense of the likelihood
ratio order (and consequently also for the usual stochastic and increasing
convex/concave orders), but it has a larger (smaller) variability for the
ARA1 model than for the ARD1 model in the sense of the log-concave order. 

Based on the previous results, if $A\left(  \cdot\right)  $ is concave, we have
$Z^{(2)}_{t}-Z^{(2)}_{nT}\prec_{lc}Y_{t}^{(1)}-Y_{nT}^{(1)}$ but $Y^{(1)}_{nT}\prec_{lc}Z^{(2)}_{nT}$ (for 
$nT\leq t<\left(n+1\right)T$). There consequently is no real hope that $Y^{(1)}_{t}\prec_{lc}Z^{(2)}_{t}$. When $A\left(\cdot\right)$ is convex, we have both
$Y^{(1)}_{t}-Y^{(1)}_{nT}\prec_{lc}Z^{(2)}_{t}-Z^{(2)}_{nT}$ and $Y^{(1)}_{nT}\prec_{lc}Z^{(2)}_{nT}$, and
$Y^{(1)}_{t}\prec_{lc}Z^{(2)}_{t}$ might be valid. However, remembering that the
log-concave order is not closed under convolution, see \cite{Whitt1985}, the
question deserves to be further studied, {\black which is done in the following example}. 
\begin{example} 
		The function $\log\left(f_{Y^{(1)}_{t}}/f_{Z^{(2)}_{t}}\right)$ is plotted in Figures \ref{logdensityconcave} and \ref{logdensityconvex} for $\rho_1=0.5$, $\rho_2=0.4$, $T=1$, $A\left(t\right)=t^{\alpha}$, $n=10$, $b=1$ at time $t=10.2$, with $\alpha=0.75$ and
$\alpha=1.25$, respectively. We observe that we do not have $Y^{(1)}_{t}\prec_{lc}Z^{(2)}_{t}$ neither
for $\alpha=0.75$ (as expected) nor for $\alpha=1.25$ for which the question
was  open.  As a conclusion, in a general setting, $Y^{(1)}_{t}$ and $Z^{(2)}_{t}$ are not comparable with respect to the log-concave order. 
\begin{figure}[tbp]
\begin{subfigure}{0.5\textwidth}
\includegraphics[width=0.9\textwidth]{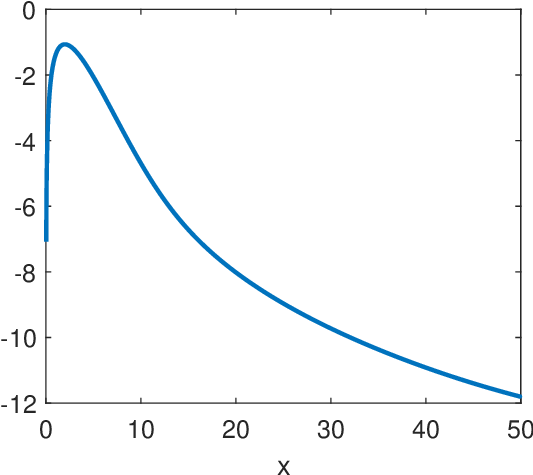}
\caption{$log(f_{Y_t^{(1)}}/f_{Z_t^{(2)}})$ for $\alpha=0.75$} \label{logdensityconcave}
\end{subfigure}%
\begin{subfigure} {0.5\textwidth}
\includegraphics[width=0.9\textwidth]{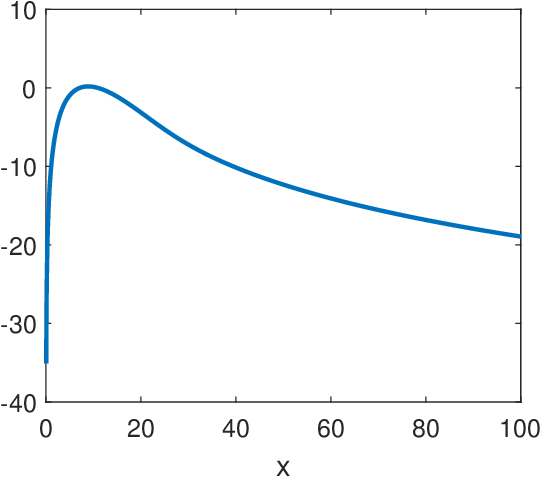}
\caption{$log(f_{Y_t^{(1)}}/f_{Z_t^{(2)}})$ for $\alpha=1.25$} \label{logdensityconvex}
\end{subfigure}
\caption{Plots of $log(f_{Y_t^{(1)}}/f_{Z_t^{(2)}})$ for $A(t)=t^{\alpha}$}
\end{figure}
\end{example}
Next result provides some conditions under which $Y_{t}^{(1)}$ and $%
Z_{t}^{(2)}$ are comparable with respect to either the increasing convex or
concave order.

\begin{theorem}
\label{Thm icx icv}If $A\left( \cdot \right) $ is concave and 
\begin{equation}
A\left( (1-\rho _{2}) t\right) \geq (1-\rho _{1})A(t)\text{ for all }t>0 \label{eqA3}
\end{equation}%
(condition of Proposition \ref{means}, known to be true if $\rho _{1}\geq
\rho _{2}$), then $Y_{t}^{(1)}\prec _{icx}Z_{t}^{(2)}$ for all {\black $t\geq 0$ (and all $T\geq 0$)}. 

If $A\left( t\right) $ is convex with a reversed inequality in $\left( \ref%
{eqA3}\right) $, then $Z_{t}^{(2)}\prec _{icv}Y_{t}^{(1)}$ for all {\black $t\geq 0$ (and all $T\geq 0$)}.
\end{theorem}
\begin{proof}
Writing $Y_{t}^{\left( 1\right) }=\left( Y_{t}^{\left( 1\right)
}-Y_{nT}^{\left( 1\right) }\right) +Y_{nT}^{\left( 1\right) }$ (the same for 
$Z_{t}^{\left( 2\right) }$), the results are direct consequences from
Propositions \ref{Prop Icx Icv} and \ref{Prop_comp_incr}, based on the fact
that both increasing convex and concave orders are stable {\black under}
convolution \cite[Thm 4.A.8 (d) page 186]{shaked2007}.
\end{proof}

\begin{example}[\black \textbf{Increasing convex order}]  We consider $t=10.5$, $T=1$, $b=1$, {\black $\rho_1$ = $\rho_2$= $0.75$ with $A(t)=t^{0.7}$ as a first case (concave case with condition $(\ref{eqA3})$ fulfilled)} and $A(t)=t^{1.1}$ as a second case (convex case with reversed condition $(\ref{eqA3})$ fulfilled). The difference 
$$%
\int_{x}^{+\infty }\bar{F}_{Z_{t}^{(2)}}\left( u\right) ~du-\int_{x}^{+\infty }%
\bar{F}_{Y_{t}^{(1)}}\left( u\right) ~du$$ is plotted in Figures \ref{Fig10a} (first case)  and \ref{Fig10b} (second case). As expected, the difference remains positive in the first case, which means that {\color{black} $Y_{t}^{(1)}\prec_{icx}Z_{t}^{(2)}$} (see (\ref{icx})). We observe that it changes sign in the convex case, which shows that $Z_{t}^{(2)}$ and $Y_{t}^{(1)}$ are not comparable with respect to the increasing convex order.

\begin{figure}[tbp]
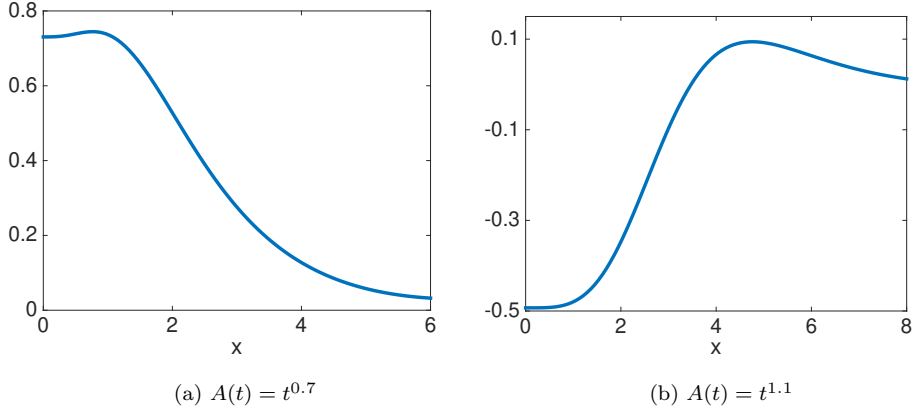

\begin{subfigure}{0.5\textwidth}
\includegraphics[width=0.9\textwidth]{Fig10a.eps}
\caption{$A(t)=t^{0.7}$} \label{Fig10a}
\end{subfigure}%
\begin{subfigure} {0.5\textwidth}
\includegraphics[width=0.9\textwidth]{Fig10b.eps}
\caption{$A(t)=t^{1.1}$} \label{Fig10b}
\end{subfigure}
\caption{Plots of $\int_{x}^{\infty} \bar{F}_{Z_t^{(2)}}(u)du-\int_{x}^{\infty}\bar{F}_{Y_t^{(1)}}(u) du$ for $A(t)=t^{\alpha}$}
\end{figure}
\end{example}
\begin{example}[\black \textbf{Increasing concave order}] We consider $t=10.5$, $T=1$, $\rho_1=0.8$, $\rho_2=0.78$, $b=1$ with $A(t)=t^{0.9}$ (concave case with condition $(\ref{eqA3})$ fulfilled) and $A(t)=t^{1.1}$ (convex case with reversed condition $(\ref{eqA3})$ true). The difference  $\int_{0}^{x} F_{Z_t^{(2)}}(u)du-\int_{0}^{x}F_{Y_t^{(1)}}(u) du$ is plotted in Figures \ref{Fig11a} (first case) and \ref{Fig11b} (second case). We observe that, as expected,  $Z_{t}^{(2)}\prec _{icv}Y_{t}^{(1)}$ (see (\ref{icv})) in the second case whereas $Z_{t}^{(2)}$ and $Y_{t}^{(1)}$ are not comparable with respect to the increasing concave order in the first case.
\begin{figure}[tbp]
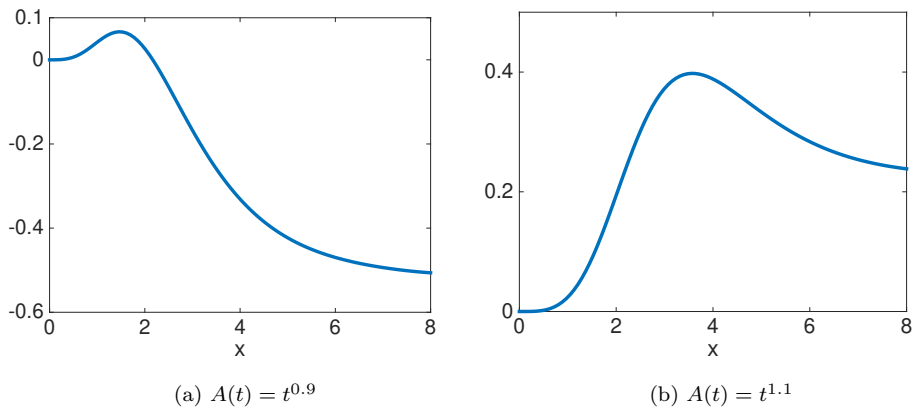

\begin{subfigure}{0.5\textwidth}
\includegraphics[width=0.9\textwidth]{figure11a.eps}
\caption{$A(t)=t^{0.9}$} \label{Fig11a}
\end{subfigure}%
\begin{subfigure} {0.5\textwidth}
\includegraphics[width=0.9\textwidth]{figure11b.eps}
\caption{$A(t)=t^{1.1}$} \label{Fig11b}
\end{subfigure}
\caption{Plots of $\int_{0}^{x} F_{Z_t^{(2)}}(u)du-\int_{0}^{x}F_{Y_t^{(1)}}(u) du$ for $A(t)=t^{\alpha}$}
\end{figure}

\end{example}

\medskip
Finally, we end this section by considering the case of a homogeneous gamma process ($A(t)=at$). 
	\begin{corollary}\label{Cor Hom}
		Assume that $A(t)=at$ for all $t \geq 0$, where $a>0$. We have the following results:
		\begin{enumerate}
			\item If $\rho_{1}\geq \rho_{2}$, then $Y_{t}^{(1)}\prec _{icx}Z_{t}^{(2)}$ and hence $\mathbb{E}(Y_{t}^{(1)})\leq \mathbb{E}(Z_{t}^{(2)})$ for all $t\geq 0$;
			\item If $\rho_{1}\leq \rho_{2}$, then $Z_{t}^{(2)}\prec _{icv}Y_{t}^{(1)}$ and hence $\mathbb{E}(Z_{t}^{(2)})\leq \mathbb{E}(Y_{t}^{(1)})$ for all $t\geq 0$;
			\item If $\rho_{1}=\rho_{2}$, then $Y_{t}^{(1)}\prec _{cx}Z_{t}^{(2)}$ and $Z_{t}^{(2)}\prec _{cv}Y_{t}^{(1)}$ and hence $\mathbb{E}(Z_{t}^{(2)})= \mathbb{E}(Y_{t}^{(1)})$ for all $t\geq 0$;
			\item $Var\left( Z_{t}^{(2)}\right) \geq Var\left( Y_{t}^{(1)}\right)$ for all $t>0$ if and only if $1-\rho_{2}\geq (1-\rho_{1})^{2}$.
		\end{enumerate} 
	\end{corollary}
\begin{proof}
	Points 1, 2 and 4 are direct consequences of Theorem \ref{Thm icx icv} and Corollary \ref{corollaryweibullexpect}. As for point 3, assume that $\rho_{1}=\rho_{2}$. This entails that  $\mathbb{E}%
	 \left(  Y_{t}^{(1)}\right)  =\mathbb{E}\left(  Z_{t}^{(2)}\right)  $. Based on the first two points, the result follows from \cite[Theorem 4.A.35 page 197]{shaked2007}.
\end{proof}
\section{Mostly equivalent imperfect repair models}
In an applied context, parameters $\rho _{1}$ and $\rho _{2}$ for ARD1 and
ARA1 models, respectively, will be estimated from feedback data, which will
be typically gathered at maintenance times $iT$, $i\geq 1$. As a
consequence, we can expect that the estimated parameters $\hat{\rho}_{1}$
and $\hat{\rho}_{2}$ should be such that the corresponding expected
deterioration levels should be very similar at maintenance times, namely
such that%
\begin{equation*}
\mathbb{E}\left( Y_{iT}^{\left( 1\right) }\right) =\frac{\left( 1-\hat{\rho}%
_{1}\right) A\left( iT\right) }{b}\simeq \mathbb{E}\left( Z_{iT}^{\left(
2\right) }\right) =\frac{A\left( \left( 1-\hat{\rho}_{2}\right) iT\right) }{b%
}.
\end{equation*}%
There hence is a specific interest for the applications to compare the ARD1
and ARA1 models under the condition 
\begin{equation}
\left( 1-\rho _{1}\right) A\left( iT\right) =A\left( \left( 1-\rho
_{2}\right) iT\right) \text{ for }i\geq 1,  \label{req}
\end{equation}%
on $(\rho_{1},\rho_{2})$, which will lead to mostly equivalent deterioration levels (at least at
maintenance times). However, the previous requirement $\left( \ref{req}%
\right) $ does not seem to have a solution for a general shape function $%
A\left( \cdot \right) $. We hence restrict the study to the power law case
$A\left( t\right) =\alpha t^{\beta }$ (with $\alpha ,\beta >0$), for
which $\left( \ref{req}\right) $ is just equivalent to 
\begin{equation}
1-\rho _{1}=\left( 1-\rho _{2}\right) ^{\beta }.  \label{Eq rho 1 rho2}
\end{equation}

This section is hence devoted to this specific power law case with the
previous relationship between $\rho _{1}$ and $\rho _{2}$, which ensures
that 
\begin{equation*}
\mathbb{E}\left( Y_{iT}^{\left( 1\right) }\right) =\mathbb{E}\left(
Z_{iT}^{\left( 2\right) }\right) \text{ for all }i\geq 1.
\end{equation*}

\begin{remark}
This specific ``equivalent'' case has a similar spirit to that detailed in 
\cite[Property 4]{doyen2004}, where the authors match the minimal wear
intensities of two imperfect repair models for recurrent events, based on
the reduction of either virtual age or failure intensity.
\end{remark}
In case of a homogeneous gamma process ($\beta=1$), the equivalent case corresponds to identical repair efficiencies for both ARD1 and ARA1 models ($\rho_1=\rho_2$), which has already been studied in Corollary \ref{Cor Hom} (point 3). We now investigate the case of a general $\beta$.

In the equivalent case, there is equality in $\left( \ref{eqA1}\right) $, $%
\left( \ref{eqA2}\right) $ and $\left( \ref{eqA3}\right) $. Based on the
fact that $A\left( \cdot \right) $ is concave (convex) when $\beta \leq
\left( \geq \right) 1$, we directly get the following results from
Proposition \ref{Prop Icx Icv}, Theorem \ref{Thm icx icv} and Corollary \ref%
{corollaryweibullexpect}.

\begin{corollary}
\label{CorEquivalent}Assume that $A\left( t\right) =\alpha t^{\beta }$ (with $%
\alpha ,\beta >0$) and that $(\rho _{1},\rho _{2})$ fulfills $\left( \ref%
{Eq rho 1 rho2}\right) $ (equivalent case). Then:

\begin{enumerate}
\item $Z_{nT}^{(2)}\prec _{cv}Y_{nT}^{(1)}$ and $Y_{nT}^{(1)}\prec
_{cx}Z_{nT}^{(2)}$ (which both entail that $var\left( Z_{nT}^{\left(
2\right) }\right) \geq var\left( Y_{nT}^{\left( 1\right) }\right) $, see 
\cite[(3.A.4) page 110]{shaked2007}) for all $n\geq 1$.

\item If $\beta \leq \left( \geq \right) 1$, then $Y_{t}^{(1)}\prec
_{icx}\left( \succ _{icv}\right) Z_{t}^{(2)}$ (which entails that $\mathbb{E}%
\left( Y_{t}^{\left( 1\right) }\right) \leq \left( \geq \right) \mathbb{E}%
\left( Z_{t}^{\left( 2\right) }\right) $) for all $t\geq 0$.

\item If $\beta \leq 1$, then $Var\left( Z_{t}^{(2)}\right) \geq Var\left(
Y_{t}^{(1)}\right) $ for all $t\geq 0$.
\end{enumerate}
\end{corollary}
{\black 
	\begin{remark}
	Based on the previous result, we can see that even if the two imperfect repair models provide similar expected deterioration levels at maintenance times, there are differences in both their location and spread between the repairs. For instance, considering $\beta \leq 1$, a possible by-product of $Y_{t}^{(1)}\prec_{icx} Z_{t}^{(2)}$ is that $\mathbb{E}((Y_{t}^{(1)}-h)^{+})\leq \mathbb{E}((Z_{t}^{(2)}-h)^{+})$ for any $h>0$, see, e.g., \cite[(4) p 182]{shaked2007}. If $h$ is a critical deterioration level in an application context, $(Y_{t}^{(1)}-h)^{+}$ and $(Z_{t}^{(2)}-h)^{+}$ correspond to the hazardous part of deterioration (beyond the critical level) and this means that the expected ``risk'' is lower for the ARD1 model than for the ARA1 one.
	\end{remark}}
	
Note that Corollary \ref{CorEquivalent} does not provide any insight for the comparison
of the variances at time $t$ when $\beta >1$, which
hence deserves a different analysis on which we now focus.

\begin{proposition}
\label{propositionvariance}Let $\beta >1$ and let 
\begin{equation}
g\left( x\right) =x^{\beta }-\left( x-\rho _{2}\right) ^{\beta }-\left(
1-(1-\rho _{2})^{2\beta }\right)  \label{g}
\end{equation}%
for $x\in \lbrack 1,2)$.

\begin{itemize}
\item If $g\left( 2^{-}\right) \leq 0$, then 
\begin{equation}
Var\left( Y_{t}^{(1)}\right) \leq Var\left( Z_{t}^{(2)}\right)
\label{Ineq VAR}
\end{equation}%
for all $t\geq 0$.

\item If $g\left( 2^{-}\right) >0$, there exists one single $x^{\ast }\in
(1,2)$ such that $g\left( x^{\ast }\right) =0$. Also:

\begin{itemize}
\item Inequality $\left( \ref{Ineq VAR}\right) $ is true for all $t\geq
t^{\ast }=\left\lceil \frac{1}{x^{\ast }-1}\right\rceil T$, where $%
\left\lceil \cdot \right\rceil $ stands for the ceiling function;

\item For each $n<\frac{1}{x^{\ast }-1}$, inequality $\left( \ref{Ineq VAR}%
\right) $ is true for all $t\in \lbrack nT,x^{\ast }nT)$, with a reversed
inequality for $t\in \lbrack x^{\ast }nT,\left( n+1\right) T)$.
\end{itemize}
\end{itemize}
\end{proposition}

\begin{proof}
For $nT\leq t<(n+1)T$, it is easy to check that $\left( \ref{Ineq VAR}%
\right) $ is equivalent to%
\begin{equation*}
t^{\beta }-\left( 1-(1-\rho _{2})^{2\beta }\right) (nT)^{\beta }\leq \left(
t-\rho _{2}nT\right) ^{\beta },
\end{equation*}%
which can also be written as $g\left( \frac{t}{nT}\right) \leq 0$ where $1 \leq \frac{t}{nT} \leq 1+\frac{1}{n}$ and where $g$
is defined by $\left( \ref{g}\right) $ for $x\in \bigcup_{n\geq1}[1,1+\frac{1}{n})=[1,2)$.

As $\beta >1$, the function $g$ increases from $g\left( 1\right) $ to $%
g\left( 2^{-}\right) $ and based on the fact that $var\left( Z_{nT}^{\left(
2\right) }\right) \geq var\left( Y_{nT}^{\left( 1\right) }\right) $ (first
point of Corollary \ref{CorEquivalent}), we have $g\left(
1\right) \leq 0$. As for the sign of 
\begin{equation*}
g\left( 2^{-}\right) =2^{\beta }-\left( 2-\rho _{2}\right) ^{\beta }-\left(
1-(1-\rho _{2})^{2\beta }\right) ,
\end{equation*}%
there are two possibilities, which lead to the following cases:

\begin{itemize}
\item If $g\left( 2^{-}\right) \leq 0$, then $g\left( x\right) \leq 0$ for
all $x\in \lbrack 1,2)$ and $\left( \ref{Ineq VAR}\right) $ is true for all $%
t\geq 0$.

\item If $g\left( 2^{-}\right) >0$, there exists one single $x^{\ast
}\in (1,2)$ such that $g\left( x^{\ast }\right) =0$, with $g\left( x\right)
<0$ for all $x\in \lbrack 1,x^{\ast })$ and $g\left( x\right) >0$ for all $%
x\in (x^{\ast },2)$.

\begin{itemize}
\item If $1+\frac{1}{n}\leq x^{\ast }$ (namely $n\geq \frac{1}{x^{\ast }-1}$%
), then $g\left( x\right) <0$ for all $x\in \lbrack 1,1+\frac{1}{n})$ and $%
\left( \ref{Ineq VAR}\right) $ is true for all $t\in \lbrack nT,\left(
n+1\right) T)$. This inequality is hence true for all $t\geq n^{\ast }T$
with $n^{\ast }=\left\lceil \frac{1}{x^{\ast }-1}\right\rceil $.

\item If $1+\frac{1}{n}>x^{\ast }$ (namely $n<\frac{1}{x^{\ast }-1}$), then $%
g\left( x\right) <0$ for all $x\in \lbrack 1,x^{\ast })$ and $\left( \ref%
{Ineq VAR}\right) $ is true for all $t\in \lbrack nT,x^{\ast }nT)$, with a
reversed inequality for $t\in \lbrack x^{\ast }nT,\left( n+1\right) T)$.
\end{itemize}
\end{itemize}
\end{proof}

\begin{remark}
In the case where $g\left( 2^{-}\right) >0$, note that the inequality $n<\frac{1}{%
x^{\ast }-1}$ is always valid for $n=1$ so that for small $n$, the
difference $Var\left( Y_{t}^{(1)}\right) -Var\left( Z_{t}^{(2)}\right) $
will always cross 0 (from - to +) on $[nT,\left( n+1\right) T)$ (and will
remain negative for larger $n$).
\end{remark}

\section{Maintenance strategies}
{\black\subsection{The reward function}\label{Sub Reward}
This section is devoted to the analysis of maintenance strategies considering the two types of repair. In all the section, the system is assumed to provide} a reward which decreases when the deterioration level of the system increases. Based
on classical functions used in the insurance literature (\cite{rolski1998}),
we assume that the reward function is given by 
\begin{equation}
g(x)=\left(b_{1}-k_1e^{\alpha _{1}x}\right)\mathbf{1}_{\left\{0\leq x\leq c\right\}}+
\left(b_{2}-k_2e^{\alpha _{2}x}\right) \mathbf{1}_{\left\{c < x\right\}},
\label{reward}
\end{equation}%
with $b_{1},b_{2},\alpha_1,\alpha _{2},k_1,k_2,c>0$ {\black and $x\geq 0$,  where $g(x)$ stands for the unitary reward per unit time when the degradation level of the system is $x$}. The function $g$ is supposed to be continuous and positive on $[0,c)$, which implies that 
\begin{equation}
b_{2}-k_2e^{\alpha _{2}c}=b_{1}-k_1e^{\alpha _{1}c}>0. 
\label{cond b}
\end{equation}
Also, we assume that $\alpha _{1}\leq \alpha _{2}$ and $k _{1}\leq k_{2}$ so that level $c$ appears as a critical level, from which the system becomes less performing.  \color{black}

With the previous assumptions, it is easy to check that $g$ is a concave
function and that $g(x)>0$ if and only if $x<L=\frac{\ln (b_{2}/k_2)}{\alpha _{2}%
}$. \color{black} Level $L$ hence appears as a critical threshold.

An example of reward function is plotted in Figure \ref{funciong} with parameters $\alpha_1=0.1$, $b_1=11$ {\black monetary units per time unit (m.t.u.)}, $\alpha_2=0.25$, $k_1=1$ m.t.u., $k_2=1$ m.t.u., $c=4$ {\black time units (t.u.)}, and $b_2$ is obtained through (\ref{cond b}). With this dataset $L\simeq 10.0144$.

\begin{figure}[tbp]
\begin{center}
\includegraphics[width=0.6\textwidth]{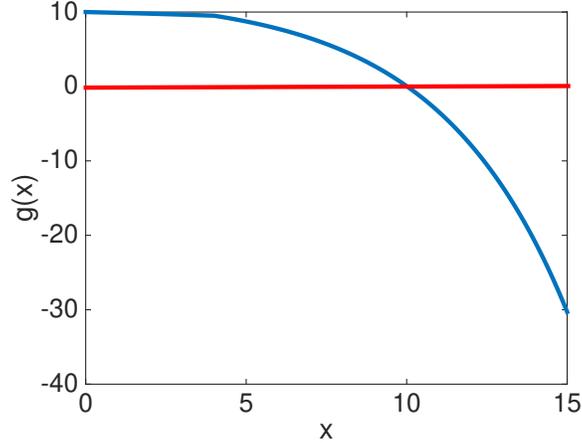}
\caption{Reward function for $\alpha_1=0.1$, $b_1=11$, $\alpha_2=0.25$, $k_1=1$, $k_2=1$, $c=4$} \label{funciong}
\end{center}
\end{figure}

{\color{black} Two different maintenance strategies are envisioned in the two following subsections. In each case, the system is put into operation at time $t=0$ and it degrades according to a non-homogeneous gamma process with parameters $A(t)$ and $\beta$. For each maintenance strategy, the two imperfect repair models are envisioned (ARA1 or ARD1) and the comparison between the two types of repair is performed through their corresponding expected reward (profit) rates per unit time on a long time run. }

{\black 
\subsection{$(n,T)$ policy}
Starting from $n \in \mathbb{N}^{*}$ and $T>0$, the $(n,T)$ maintenance scheme is developed as follows:
\begin{itemize}
\item Imperfect repairs based on either one of the two models (ARD1 or ARA1) are performed at times $T, 2T, 3T, \ldots$
\item Each imperfect repair costs $C_r$ monetary units (m.u.),
\item The profit per unit time is given by the reward function $g$ from  (\ref{reward}),
\item The system is replaced by a new one at the time of the $n$-th imperfect repair ($nT$) with a cost of $C$ m.u..
\end{itemize}
 Based on the  renewal reward theorem, see, e.g., \cite{tijms2003}, the long time reward rate per unit time for this policy is given by:
\begin{equation}  \label{nTARD}
R_{ARD}(n,T)=\frac{\int_{0}^{nT}\mathbb{E}\left(g(Y_s^{(1)})\right)ds-(n-1)C_r-C}{nT}
\end{equation}
when ARD1 repairs are considered and
\begin{equation}  \label{nTARA}
R_{ARA}(n,T)=\frac{\int_{0}^{nT}\mathbb{E}\left(g(Z_s^{(2)})\right)ds-(n-1)C_r-C}{nT}
\end{equation}
for ARA1 repairs, where $Y_s^{(1)}$ and $Z_s^{(2)}$ are is given by (\ref{Yt}) and (\ref{Zt}), respectively, with $\rho$ substituted by $\rho_i,i=1,2$ .

\begin{remark}\label{Rem Appli}
	Based on the fact that the function $-g(\cdot)$ is increasing and convex, the previously obtained theoretical results allow to derive several observations on the reward rates:
	\begin{itemize}
		\item From Propositions \ref{Prop Y rho} and \ref{Prop Z rho}, we get that $\mathbb{E}\left(-g(Y_t^{(1)})\right)$ and $\mathbb{E}\left(-g(Z_t^{(2)})\right)$  decrease with $\rho_1$ and $\rho_2$ for all $t>0$, respectively. This entails that the objective profit rates  $R_{ARD}(n,T)$ and $R_{ARA}(n,T)$ increase with the effectiveness of the repair ($\rho_1$ and $\rho_2$, respectively) for both ARD1 and ARA1 models.
		\item If the shape function $A(\cdot )$ is concave and $A((1-\rho_2)t) \geq (1-\rho_1)A(t)$ for all $t>0$ (Condition $(\ref{eqA3})$), then Theorem \ref{Thm icx icv} entails that
		$$\mathbb{E}\left(g(Z_t^{(2)})\right) \leq \mathbb{E}\left(g(Y_t^{(1)})\right), \quad \forall t>0, $$ from which we derive that the objective profit functions for the two repair models are comparable:
		\begin{equation*}
		R_{ARA}\left( n,T\right) \leq R_{ARD}\left(n,T\right) \text{ for all } n,T.
		\end{equation*}
	\end{itemize}
\end{remark}
}

We now come to some numerical illustrations.
\begin{example}
The parameters of the gamma process are $A(t)=t^{0.5}+t^{0.75}$ {\black (concave function)} and $b=1$; those for the reward reward function $g$ are $\alpha_1=0.1$, $\alpha_2=0.25$, $k_1=1$, $k_2=1$, $b_1=11$, $c=4$, which implies that $b_2=12.2265$. The $(n,T)$ policy is considered with $C_r=2$ m.u. as cost of imperfect repair and $C=25$ m.u. as replacement cost.

Figure \ref{FigR} shows the operating profit rate given
in (\ref{nTARD}) for both ARD1 and ARA1 models using $\rho_1=\rho_2=0.5$ (for which inequality $(\ref{eqA3})$ is true). The computations have been made using 8 points  for $T $ from 1 to 6 and 10 points for $n$ from 1 to 10 with 5000 simulations in
each point. The Simpson method is applied for the integrals in (\ref{nTARD}),  with 20 points from $0$ to $nT$.
\begin{figure}[t!]
	\hspace*{-0.75cm}	
	\begin{subfigure}[t]{0.65\textwidth}
		%\centering
		\includegraphics[width=0.825\textwidth]{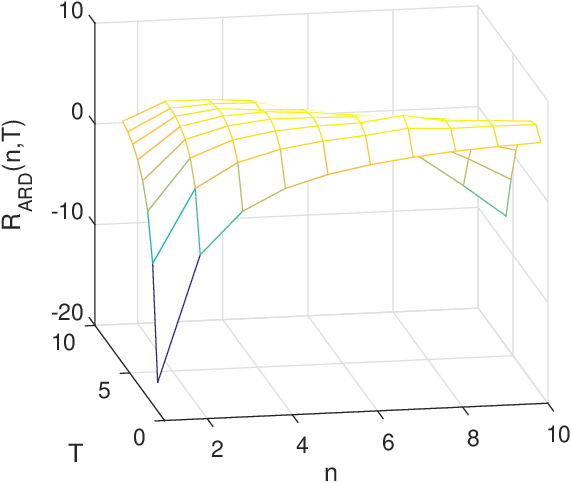}
		\caption{ARD1 model}  
	\end{subfigure}%
	~ \hspace*{-1.5cm}
	\begin{subfigure}[t]{0.65\textwidth}
		%\centering
		\includegraphics[width=0.825\textwidth]{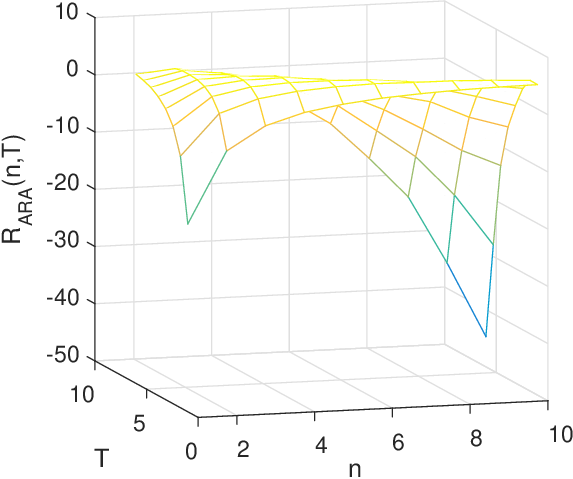}
		\caption{ARA1 model} 
	\end{subfigure}
	\caption{Operational profit rate, $(n,T)$ policy}
	\label{FigR}
\end{figure}

Under an ARD1 model, the optimal profit rate is obtained for $(n,T)=(7,2.43)$  with a profit rate of $R_{ARD}(7,2.43)=6.85$ m.u. per unit time. Under an ARA1 model, the optimal profit rate is obtained for $(n,T)=(4,3.14)$ with a profit rate of $R_{ARA}(4,3.14)=5.29$ m.u. per unit time. Figure \ref{differencesnT} shows the
difference of the profit rates under the two repair models, that is, $%
R_{ARD}(n,T)-R_{ARA}(n,T)$ for all $n$ and $T$. As expected from the previous theoretical results, $R_{ARA}(n,T)\leq R_{ARD}(n,T)$. Also, we can observe that the difference between the two rates increases as $n$ and $T$ increases.  

\begin{figure}[t!]
	\hspace*{-0.75cm}	
	\begin{subfigure}[t]{0.7\textwidth}
		%\centering
		\includegraphics[width=0.75\textwidth]{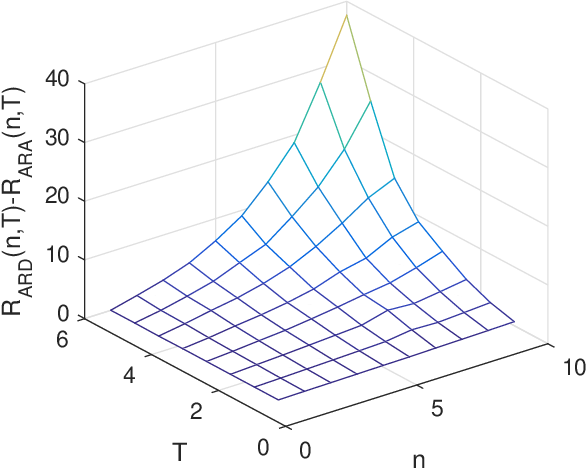}
		\caption{Case $\rho_1=\rho_2=0.5$} \label{differencesnT}	
	\end{subfigure}%
	~ \hspace*{-2cm}
	\begin{subfigure}[t]{0.7\textwidth}
		%\centering
		\includegraphics[width=0.75\textwidth]{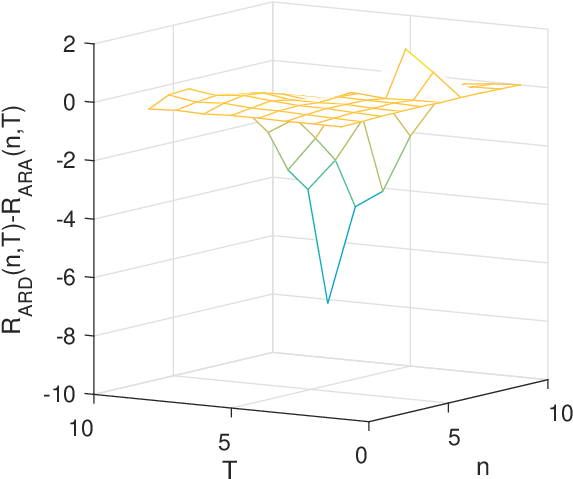}
		\caption{Case $\rho_1=0.31$ and $\rho_2=0.5$} \label{ARDAARAcont}
	\end{subfigure}
	\caption{Differences between the operational profit rates, $(n,T)$ policy}
\end{figure}
\end{example}
\begin{example}
{\black Keeping the same parameters as in the previous example except from the repair efficiency for the ARD1 model which becomes} $\rho_1=0.31$, Figure \ref{ARDAARAcont} shows the difference {\black between} the profit rates under the two repair models. Here $A(t)=t^{0.5}+t^{0.75}$ is concave but condition (\ref{eqA3}) of Theorem \ref{Thm icx icv} is not valid any more. We observe that there is no dominance of the profit rate of one model {\black over} the other.  
\end{example}

{\color{black} Although the $(n,T)$ policy allows us to compare the two models of repair, under this maintenance policy}, the system is (imperfectly) repaired even when the
system is so degraded that the reward {\black has become negative}. We now suggest a more
realistic {\black condition-based maintenance strategy, where the maintenance action depends on } the degradation level of the system. 

{\color{black}
\subsection{(M,T) policy}
 Let $T>0$ and let $M \in [0,L)$ be a preventive maintenance threshold. (We recall that $L$ is the critical threshold defined in Subsection \ref{Sub Reward} from where the reward becomes negative). The $(M,T)$ condition-based maintenance scheme is developed as follows:
\begin{itemize}
\item The system is inspected at times $T, 2T, 3T, \ldots$ and the system degradation level is checked. 
\item By an inspection:
	\begin{itemize}
		\item If the degradation level does not exceed the threshold $M$, an imperfect repair based on either one of the two models (ARD1 or ARA1) is performed with a cost of $C_r$ m.u.;
		\item If the degradation level is between levels $M$ and $L$, a preventive replacement is performed and the system is instantaneously replaced by a new one with a cost of $C_p$ m.u.; 	
		\item If the degradation level exceeds $L$, a instantaneous corrective replacement takes place with a cost of $C_c$ m.u..
	\end{itemize}
\item The profit per unit time is given by the reward function $g$ from (\ref{reward}).
\end{itemize}
}

The successive (corrective or preventive) replacements of the system appear as 
the points of a renewal process, and the long time profit rate per unit time is given by 
\begin{eqnarray} \label{CARDTM}
&&C_{ARD}(T,M) \\ \nonumber
&=&\frac{\mathbb{E}\left( \int_{0}^{R}g\left( Y_s^{(1)}\right) ds-C_{r}\left(
[R/T]-1\right) -C_{p}\mathbf{1}_{\left\{ M\leq Y_{R}^{(1)}<L\right\} }-C_{c}%
\mathbf{1}_{\left\{ L\leq Y_{R}^{(1)}\right\} }\right) }{\mathbb{E}(R)}
\end{eqnarray}%
for the ARD1 model, with a similar expression for the ARA1 model ($C_{ARA}(T,M)$), 
where $R$ stands for the time to a system replacement and $g$ denotes the reward function given by (\ref{funciong}). Due to the complexity of the {\color{black} $(M,T)$ policy}, there is no hope here to find {\color{black} analytical conditions that could ensure the dominance of one function} $C_{ARD}(T,M)$ or $%
C_{ARA}(T,M)$ over the other. Their comparison is hence made on a numerical example.

\begin{example}
	The parameters of the gamma process are $A(t)=1.3 t$ and $b=0.8$. For the reward function $g$, they are $\alpha_1=0.4$, $\alpha_2=0.5$, $b_1=800$, $k_1=1.05$, $k_2=1.07$, $c=8$, which implies $b_2=832.6609$ and $L=13.3139$. The repair efficiencies of the ARD1/ARA1 repairs are $\rho_{1}=\rho_{2}=0.9$. Their common cost is $C_{r}=200$ m.u.. The cost of a preventive replacement is $C_p=1000$ m.u. whereas it is $C_c=1300$ m.u. for a corrective one.
	Figure \ref{figureCBMarda} shows the profit rate for the maintained
	system under an ARD1 repair. The optimal maintenance strategy is
	obtained for $(M,T)=(9.21,3.05)$ with a profit rate of $%
	C_{ARD}(3.05, 9.21)=673.94$ m.u. per unit time. Figure \ref%
	{figureCBMara} shows the profit rate for the ARA1 repairs. The optimal maintenance strategy is obtained
	for  $(M,T)=(10.24,3.05)$ with a profit rate of  $%
	C_{ARA}(3.05,10.24)=684.34$ m.u. per unit time. These figures have
	been computed considering a grid of 10 points for $T$ from  1.14 to 4 and a
	grid of  13 \color{black} points for $M$ from 1 to $L$ and  10000 \color{black} simulations for each pair of points.
	
	\begin{figure}[t!]
		\hspace*{-0.75cm}	
		\begin{subfigure}[t]{0.68\textwidth}
			%\centering
			\includegraphics[width=0.8\textwidth]{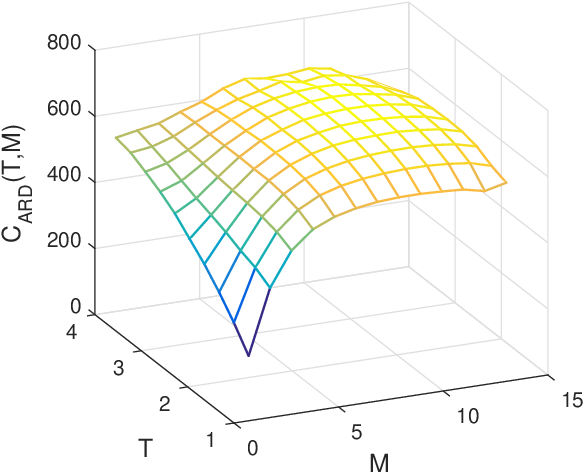}
			\caption{$\rho_1=0.9$, ARD1 model} \label{figureCBMarda}	
		\end{subfigure}%
		~ \hspace*{-1.75cm}
		\begin{subfigure}[t]{0.68\textwidth}
			%\centering
			\includegraphics[width=0.8\textwidth]{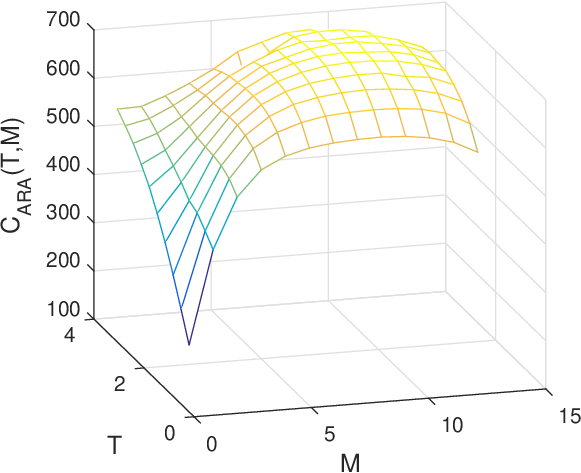}
			\caption{$\rho_2=0.9$, ARA1 model} \label{figureCBMara}	
		\end{subfigure}
		\caption{Operational profit rates, $(M,T)$ policy}
	\end{figure}
	
	Figure \ref{differencesCBM} shows the difference between the profit rates \linebreak $C_{ARD}(T,M)~-~C_{ARA}(T,M)$ for this dataset.  Although conditions of Theorem \ref{Thm icx icv} are fulfilled, we can see that the sign of the difference changes, so that there is no dominance of the profit rate of one model on the other. 
	\begin{figure}[tbp]
		\begin{center}
			\includegraphics[width=0.6\textwidth]{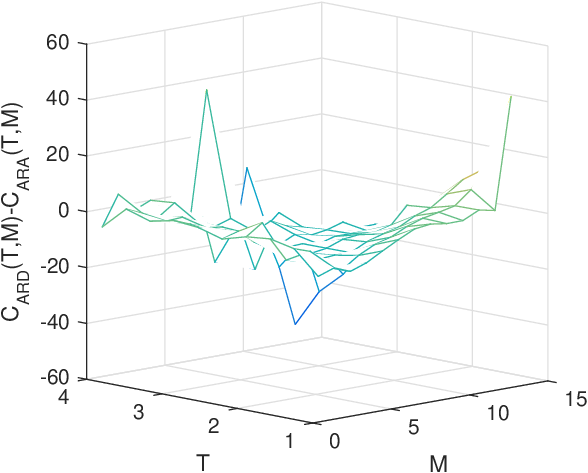}
			\caption{Difference of operational profit rates, $\rho_1=\rho_2=0.9$} \label{differencesCBM}
		\end{center}
	\end{figure}
\end{example}

\section{Conclusions and perspective}

{\black Two imperfect repair models for a degrading system are compared in this paper. The comparison is performed in terms of location and spread of the two resulting stochastic processes. Results are provided in terms of moments and likelihood ratio ordering, as well as in terms of (increasing) convex/concave ordering, which is not so common in the reliability literature. Two maintenance strategies are also  developed, which are assessed through a reward function, which takes into account the effective deterioration level of the system (the lower the deterioration level, the higher the reward), contrary to classical objective functions from the literature.  

The paper is developed under a periodic imperfect repair scheme, for sake of simplification. It is however easy to check that all the results of the paper would remain valid under a deterministic non periodic repair scheme (with the same maintenance times for both imperfect repair models), with very slight modifications. Even more, considering random maintenance times (independent on the deterioration level and identically distributed for both imperfect repair models), most results would also remain valid, such as the likelihood ratio and (increasing) convex/concave comparison results, based on the closure under mixture property of these stochastic orders \cite[Thm 1.C.15. p 48, Thm 4.A.8. p 185]{shaked2007}.

Note also that if the paper focuses on some specific stochastic orders,} other ones could also be considered such as Laplace transform or Excess Wealth orders for instance. Other questions of interest concern the comparison of remaining lifetimes, considering the system as failed (or too degraded) when its deterioration level is beyond a fixed failure (critical) threshold. From a theoretical point of view, this seems a difficult issue in a general setting. One could then look at partial results {\black in specific situations.

 Another point of interest would be to try and compare the two types of imperfect repairs dropping the gamma-process assumption. Based on the fact that, under technical conditions, normal random variables are comparable with respect to several stochastic orders (see, e.g., \cite{muller2002}), one can wonder whether it could be possible to get some similar results as in the paper for Wiener processes with drift. (Not all however, because comparison results between normal random variables in the likelihood ratio ordering sense require that the random variables share the same variance, which cannot be the case in our context. The same for the log-concave order, which requires that the two normal random variables share the same mean, see \cite{Whitt1985}). Other deteriorations processes might also be envisioned, such as inverse Gaussian or inverse Gamma processes.

Finally, other maintenance strategies could be envisioned, based on either one of the two types of imperfect repair.  

\section*{Acknowledgements} 
Both authors warmly thank the three referees and the Editor for their constructive remarks and comments, which have led to a much clearer and better structured paper.}

\bibliographystyle{elsarticle-harv}

\end{document}